\documentclass[11pt]{article}

\usepackage{array}
\usepackage{amsfonts}
\usepackage{amscd}
\usepackage{amssymb}
\usepackage{amsthm}
\usepackage{amsmath}
\usepackage{stmaryrd}
\usepackage{graphicx}
\usepackage{color}
\usepackage{verbatim}
\usepackage{mathrsfs}
\usepackage{tikz}
\usetikzlibrary{calc}
\usepackage{caption}
\usepackage[neveradjust]{paralist}

 \theoremstyle{plain}
 \newtheorem{thm1}{Theorem}
 \newtheorem{cor1}[thm1]{Corollary}
\newtheorem{thm}{Theorem}[section]
\newtheorem{lemma}[thm]{Lemma}
\newtheorem{prop}[thm]{Proposition}
\newtheorem{cor}[thm]{Corollary}

\newtheorem{fact}[thm]{Fact}
\theoremstyle{definition}
\newtheorem{defn}[thm]{Definition}
\newtheorem{remark}[thm]{Remark}
\newtheorem{example}[thm]{Example}

\numberwithin{equation}{section}

\def\sA{\mathsf{A}}
\def\sB{\mathsf{B}}
\def\sC{\mathsf{C}}
\def\sD{\mathsf{D}}
\def\sE{\mathsf{E}}
\def\sF{\mathsf{F}}
\def\sG{\mathsf{G}}

\def\sX{\mathsf{X}}

\def\cA{\mathcal{A}}

\def\cC{\mathcal{C}}

\def\FF{\mathbb{F}}

\def\KK{\mathbb{K}}

\DeclareMathOperator\type{\tau}
\DeclareMathOperator\Type{\mathrm{Typ}}
\DeclareMathOperator\Res{\mathrm{Res}}
\DeclareMathOperator\proj{\mathrm{proj}}

\DeclareMathOperator\Opp{\mathrm{Opp}}
\DeclareMathOperator\disp{\mathrm{disp}}

\def\op{\mathrm{op}}
\def\id{\mathrm{id}}
\def\<{\langle}
\def\>{\rangle}

\makeatletter
\renewcommand{\@makefnmark}{\mbox{\textsuperscript{}}}
\makeatother

\title{Opposition diagrams for automorphisms of large spherical buildings}
\author{James Parkinson 
\and
Hendrik Van Maldeghem}
\date{\today}

\begin{document}

\maketitle

\begin{abstract}
Let $\theta$ be an automorphism of a thick irreducible spherical building~$\Delta$ of rank at least~$3$ with no Fano plane residues. We prove that if there exist both type $J_1$ and $J_2$ simplices of $\Delta$ mapped onto opposite simplices by $\theta$, then there exists a type $J_1\cup J_2$ simplex of $\Delta$ mapped onto an opposite simplex by~$\theta$. This property is called \textit{cappedness}. We give applications of cappedness to opposition diagrams, domesticity, and the calculation of displacement in spherical buildings. In a companion piece to this paper we study the thick irreducible spherical buildings containing Fano plane residues. In these buildings automorphisms are not necessarily capped.
\end{abstract}


\section*{Introduction}

Let $\theta$ be an automorphism of a spherical building $\Delta$ of type $(W,S)$. The analysis of the fixed element geometry $\mathrm{Fix}(\theta)$ of $\theta$ is a powerful and well-established technique in building theory, see for example the beautiful theory of Tits indices and fixed subbuildings~\cite{PMW:15,Tit:66}. A complementary concept to fixed element theory is the ``opposite geometry'' $\mathrm{Opp}(\theta)$ consisting of all simplices of $\Delta$ that are mapped onto opposite simplices by~$\theta$. This geometry arises naturally in Curtis-Phan Theory, where it is used to efficiently encode presentations of groups acting on buildings (see \cite{BGHS:03,Gra:09}), however compared to the fixed element theory very little is known concerning $\Opp(\theta)$. In this paper we initiate a systematic analysis of the structure of the geometry~$\Opp(\theta)$.

To motivate and illustrate the key concepts in an example, let $\theta$ be a collineation of a thick $\sE_7$ building~$\Delta$, and construct the \textit{opposition diagram} of $\theta$ by encircling all nodes $s\in S$ of the Coxeter graph with the property that there exists a type~$s$ vertex in $\Opp(\theta)$. What are the possible opposition diagrams that can arise? It turns out that the number of possible diagrams is far less than the trivial bound of $2^7$. In fact it follows from our work that there are only $6$ possibilities:
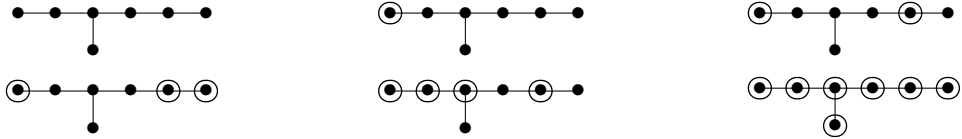
\begin{center}
\begin{tabular}{l}
\begin{tikzpicture}[scale=0.5]
\node at (0,0.3) {};
\node [inner sep=0.8pt,outer sep=0.8pt] at (-2,0) (1) {$\bullet$};
\node [inner sep=0.8pt,outer sep=0.8pt] at (-1,0) (3) {$\bullet$};
\node [inner sep=0.8pt,outer sep=0.8pt] at (0,0) (4) {$\bullet$};
\node [inner sep=0.8pt,outer sep=0.8pt] at (1,0) (5) {$\bullet$};
\node [inner sep=0.8pt,outer sep=0.8pt] at (2,0) (6) {$\bullet$};
\node [inner sep=0.8pt,outer sep=0.8pt] at (3,0) (7) {$\bullet$};
\node [inner sep=0.8pt,outer sep=0.8pt] at (0,-1) (2) {$\bullet$};
\draw (-2,0)--(3,0);
\draw (0,0)--(0,-1);
\phantom{\draw [line width=0.5pt,line cap=round,rounded corners] (1.north west)  rectangle (1.south east);}
\phantom{\draw [line width=0.5pt,line cap=round,rounded corners] (7.north west)  rectangle (7.south east);}
\end{tikzpicture}\hspace{2cm}
\begin{tikzpicture}[scale=0.5]
\node at (0,0.3) {};
\node [inner sep=0.8pt,outer sep=0.8pt] at (-2,0) (1) {$\bullet$};
\node [inner sep=0.8pt,outer sep=0.8pt] at (-1,0) (3) {$\bullet$};
\node [inner sep=0.8pt,outer sep=0.8pt] at (0,0) (4) {$\bullet$};
\node [inner sep=0.8pt,outer sep=0.8pt] at (1,0) (5) {$\bullet$};
\node [inner sep=0.8pt,outer sep=0.8pt] at (2,0) (6) {$\bullet$};
\node [inner sep=0.8pt,outer sep=0.8pt] at (3,0) (7) {$\bullet$};
\node [inner sep=0.8pt,outer sep=0.8pt] at (0,-1) (2) {$\bullet$};
\draw (-2,0)--(3,0);
\draw (0,0)--(0,-1);
\draw [line width=0.5pt,line cap=round,rounded corners] (1.north west)  rectangle (1.south east);
\end{tikzpicture}\hspace{2cm}
\begin{tikzpicture}[scale=0.5]
\node [inner sep=0.8pt,outer sep=0.8pt] at (-2,0) (1) {$\bullet$};
\node [inner sep=0.8pt,outer sep=0.8pt] at (-1,0) (3) {$\bullet$};
\node [inner sep=0.8pt,outer sep=0.8pt] at (0,0) (4) {$\bullet$};
\node [inner sep=0.8pt,outer sep=0.8pt] at (1,0) (5) {$\bullet$};
\node [inner sep=0.8pt,outer sep=0.8pt] at (2,0) (6) {$\bullet$};
\node [inner sep=0.8pt,outer sep=0.8pt] at (3,0) (7) {$\bullet$};
\node [inner sep=0.8pt,outer sep=0.8pt] at (0,-1) (2) {$\bullet$};
\draw (-2,0)--(3,0);
\draw (0,0)--(0,-1);
\draw [line width=0.5pt,line cap=round,rounded corners] (1.north west)  rectangle (1.south east);
\draw [line width=0.5pt,line cap=round,rounded corners] (6.north west)  rectangle (6.south east);
\end{tikzpicture}\\
\begin{tikzpicture}[scale=0.5]
\node at (0,0.3) {};
\node [inner sep=0.8pt,outer sep=0.8pt] at (-2,0) (1) {$\bullet$};
\node [inner sep=0.8pt,outer sep=0.8pt] at (-1,0) (3) {$\bullet$};
\node [inner sep=0.8pt,outer sep=0.8pt] at (0,0) (4) {$\bullet$};
\node [inner sep=0.8pt,outer sep=0.8pt] at (1,0) (5) {$\bullet$};
\node [inner sep=0.8pt,outer sep=0.8pt] at (2,0) (6) {$\bullet$};
\node [inner sep=0.8pt,outer sep=0.8pt] at (3,0) (7) {$\bullet$};
\node [inner sep=0.8pt,outer sep=0.8pt] at (0,-1) (2) {$\bullet$};
\draw (-2,0)--(3,0);
\draw (0,0)--(0,-1);
\draw [line width=0.5pt,line cap=round,rounded corners] (1.north west)  rectangle (1.south east);
\draw [line width=0.5pt,line cap=round,rounded corners] (6.north west)  rectangle (6.south east);
\draw [line width=0.5pt,line cap=round,rounded corners] (7.north west)  rectangle (7.south east);
\end{tikzpicture}\hspace{2cm}
\begin{tikzpicture}[scale=0.5]
\node [inner sep=0.8pt,outer sep=0.8pt] at (-2,0) (1) {$\bullet$};
\node [inner sep=0.8pt,outer sep=0.8pt] at (-1,0) (3) {$\bullet$};
\node [inner sep=0.8pt,outer sep=0.8pt] at (0,0) (4) {$\bullet$};
\node [inner sep=0.8pt,outer sep=0.8pt] at (1,0) (5) {$\bullet$};
\node [inner sep=0.8pt,outer sep=0.8pt] at (2,0) (6) {$\bullet$};
\node [inner sep=0.8pt,outer sep=0.8pt] at (3,0) (7) {$\bullet$};
\node [inner sep=0.8pt,outer sep=0.8pt] at (0,-1) (2) {$\bullet$};
\draw (-2,0)--(3,0);
\draw (0,0)--(0,-1);
\draw [line width=0.5pt,line cap=round,rounded corners] (1.north west)  rectangle (1.south east);
\draw [line width=0.5pt,line cap=round,rounded corners] (3.north west)  rectangle (3.south east);
\draw [line width=0.5pt,line cap=round,rounded corners] (4.north west)  rectangle (4.south east);
\draw [line width=0.5pt,line cap=round,rounded corners] (6.north west)  rectangle (6.south east);
\end{tikzpicture}\hspace{2cm}
\begin{tikzpicture}[scale=0.5]
\node [inner sep=0.8pt,outer sep=0.8pt] at (-2,0) (1) {$\bullet$};
\node [inner sep=0.8pt,outer sep=0.8pt] at (-1,0) (3) {$\bullet$};
\node [inner sep=0.8pt,outer sep=0.8pt] at (0,0) (4) {$\bullet$};
\node [inner sep=0.8pt,outer sep=0.8pt] at (1,0) (5) {$\bullet$};
\node [inner sep=0.8pt,outer sep=0.8pt] at (2,0) (6) {$\bullet$};
\node [inner sep=0.8pt,outer sep=0.8pt] at (3,0) (7) {$\bullet$};
\node [inner sep=0.8pt,outer sep=0.8pt] at (0,-1) (2) {$\bullet$};
\draw (-2,0)--(3,0);
\draw (0,0)--(0,-1);
\draw [line width=0.5pt,line cap=round,rounded corners] (1.north west)  rectangle (1.south east);
\draw [line width=0.5pt,line cap=round,rounded corners] (3.north west)  rectangle (3.south east);
\draw [line width=0.5pt,line cap=round,rounded corners] (4.north west)  rectangle (4.south east);
\draw [line width=0.5pt,line cap=round,rounded corners] (6.north west)  rectangle (6.south east);
\draw [line width=0.5pt,line cap=round,rounded corners] (2.north west)  rectangle (2.south east);
\draw [line width=0.5pt,line cap=round,rounded corners] (5.north west)  rectangle (5.south east);
\draw [line width=0.5pt,line cap=round,rounded corners] (7.north west)  rectangle (7.south east);
\end{tikzpicture} 
\end{tabular}
\end{center}

A fundamental result of Leeb~\cite[\S5]{Lee:00} and Abramenko and Brown~\cite[Proposition~4.2]{AB:09} states that if $\theta$ is a nontrivial automorphism of a thick spherical building then $\mathrm{Opp}(\theta)$ is necessarily nonempty, and hence the first diagram above occurs if and only if $\theta$ is the identity. For the second, third, forth, and fifth diagrams it is clear that the automorphism in question maps no chamber to an opposite chamber. Automorphisms mapping no chamber to an opposite chamber are called \textit{domestic automorphisms} (the terminology here is aligned with the thematics of the language of building theory, reflecting the idea that these automorphisms stay ``close to home''). These automorphisms have recently enjoyed extensive investigation, including the series \cite{TTM:11,TTM:12,TTM:12b} where domesticity in projective spaces, polar spaces, and generalised quadrangles is studied, \cite{HVM:12} where symplectic polarities of large $\sE_6$ buildings are classified in terms of domesticity, \cite{HVM:13} where domestic trialities of $\sD_4$ buildings are classified, and \cite{PTM:15} where domesticity in generalised polygons is studied. 

Returning to the $\sE_7$ example, if $\theta$ is not domestic then the opposition diagram of $\theta$ is necessarily the sixth of the above diagrams, with all nodes encircled. However, can this diagram be the opposition diagram of a domestic automorphism? It is a priori possible that there are vertices of each type $1,2,3,4,5,6$ and $7$ mapped onto opposite vertices, yet no chamber mapped to an opposite chamber. Such an automorphism is called \textit{exceptional domestic}. It turns out from the results of this paper that if the $\sE_7$ building $\Delta$ contains no Fano plane residues then exceptional domestic automorphisms do not exist. In contrast, we show in~\cite{PVM:17b} that if $\Delta$ is an $\sE_7$ building containing a Fano residue (thus $\Delta$ is the building of the Chevalley group $\sE_7(2)$) then $\Delta$ admits exceptional domestic automorphisms.

More generally one may ask whether the existence of both a type $J_1$ simplex and a type $J_2$ simplex in $\Opp(\theta)$ implies the existence of a type $J_1\cup J_2$ simplex in $\Opp(\theta)$. An automorphism satisfying this property is called \textit{capped}. An equivalent formulation of this concept is as follows. The \textit{type} $\Type(\theta)$ of an automorphism $\theta$ is the union of all subsets $J\subseteq S$ such that there exists a type $J$ simplex in $\Opp(\theta)$. Thus, in the above diagrams, $\Type(\theta)$ is the set of all encircled nodes. Then an automorphism $\theta$ of a spherical building is capped if and only if there exists a type $\Type(\theta)$ simplex in $\Opp(\theta)$.

The main theorem of this paper is the following. The proof is contained in Sections~\ref{sec:3} and~\ref{sec:4}.

\begin{thm1}\label{thm:main}
Let $\Delta$ be a thick irreducible spherical building of type $(W,S)$ of rank at least~$3$, and let $\theta$ be an automorphism of $\Delta$. If $\Delta$ has no Fano plane residues then $\theta$ is capped.
\end{thm1}

We call the thick irreducible spherical buildings of rank at least~$3$ with no Fano plane residues \textit{large buildings}, and those containing at least one Fano plane residue are called \textit{small buildings}. Thus Theorem~\ref{thm:main} says that every automorphism of a large spherical building is capped. Note that the small buildings are precisely the buildings $\sA_n(2)$, $\sB_n(2)$, $\sB_n(2,4)={^2}\sD_{n+1}(4)$, $\sD_n(2)$, $\sE_6(2)$, $\sE_7(2)$, $\sE_8(2)$, $\sF_4(2)$, $\sF_4(2,4)={^2}\sE_6(4)$ for some $n\geq 3$. 

Theorem~\ref{thm:main} has the following immediate corollary.

\begin{cor1}\label{cor:1}
No large building admits an exceptional domestic automorphism.
\end{cor1}

Cappedness places severe restrictions on the possible opposition diagram of an automorphism. We develop this theory in Section~\ref{sec:2} via a combinatorial approach reminiscent of Tits indices (see \cite[Part 3]{PMW:15}). In particular we show that the opposition diagrams of capped automorphisms satisfy three simple combinatorial properties, and we then use these properties to classify the possible opposition diagrams. The result is the following theorem.

\begin{thm1}\label{thm:main2}
If $\theta$ is a capped automorphism of a thick irreducible spherical building then the opposition diagram of $\theta$ appears in the corresponding Table~\ref{table:1}--Table~\ref{table:5} (where in the table $\Gamma$ is the Coxeter graph of $\Delta$, and $\pi_{\theta}$ is the automorphism of $\Gamma$ induced by $\theta$). 
\end{thm1}

Thus, for example, by Theorems~\ref{thm:main} and~\ref{thm:main2} the $6$ diagrams listed above exhaust the possible opposition diagrams for automorphisms of large $\sE_7$ buildings. More generally, Theorems~\ref{thm:main} and~\ref{thm:main2} immediately imply the following corollary. 

\begin{cor1}\label{cor:3}
Let $\theta$ be an automorphism of a large building~$\Delta$. Then the opposition diagram of $\theta$ appears in the corresponding Table~\ref{table:1}--Table~\ref{table:5}.
\end{cor1}

\begin{center}
\noindent

\captionof{table}{Type $\sA$ collineation and duality diagrams}\label{table:1}
\end{center}

\begin{center}
\noindent%
\captionof{table}{Type $\sB$/$\sC$ collineation diagrams}\label{table:2}
\end{center}

\begin{center}
\noindent%
\captionof{table}{Type $\sD$ collineation and duality diagrams}\label{table:3}
\end{center}

\begin{center}
\noindent%
\captionof{table}{Triality diagrams}\label{table:4}
\end{center}

\begin{center}
\noindent%
\captionof{table}{Diagrams of exceptional types and dihedral groups}\label{table:5}
\end{center}

We have listed the possible diagrams in the above tables up to automorphisms of the Coxeter graph~$\Gamma$. That is, if the triple $(\Gamma,J,\pi_{\theta})$ appears in the above tables (where $J$ is the set of encircled nodes), and if $\sigma$ is an automorphism of the Coxeter graph, then the triple $(\Gamma,J^{\sigma},\pi_{\theta})$ is implicit in the tables. For example, the $\sF_4$ diagram
%
 is implicit in Table~\ref{table:5}. We also emphasise that we work with Coxeter graphs rather than Dynkin diagrams in this paper.

In Section~\ref{sec:2} we show that if $\theta$ is a capped automorphism of a thick spherical building then the displacement of $\theta$ can be computed directly from the opposition diagram in a simple way (recall that the \textit{displacement} $\disp(\theta)$ of an automorphism $\theta$ is the maximum gallery distance between a chamber $C\in\Delta$ and its image $C^{\theta}$). More precisely, we prove:

\begin{thm1}\label{thm:main3}
Let $\theta$ be a capped automorphism of a thick spherical building $\Delta$ of type $(W,S)$ and let $J=\Type(\theta)$. The displacement of $\theta$ is
$$
\disp(\theta)=\mathrm{diam}(W)-\mathrm{diam}(W_{S\backslash J}),
$$ 
where $\mathrm{diam}(W)$ and $\mathrm{diam}(W_{S\backslash J})$ are the diameters of $W$ and the parabolic subgroup $W_{S\backslash J}$.
\end{thm1}

Thus, applying Theorem~\ref{thm:main} and Theorem~\ref{thm:main3} gives:

\begin{cor1}\label{cor:2}
Let $\theta$ be an automorphism of a large building $\Delta$ of type $(W,S)$ and let $J=\Type(\theta)$. The displacement of $\theta$ is
$
\disp(\theta)=\mathrm{diam}(W)-\mathrm{diam}(W_{S\backslash J}).
$ 
\end{cor1}

This paper can be seen as a natural continuation of \cite{PTM:15} where automorphisms of rank~$2$ spherical buildings were investigated, and our series of investigations continues in~\cite{PVM:17b} where we study small buildings, and show that for these buildings not all automorphisms are capped, exceptional domestic automorphisms exist, and that the above formula for displacement may fail. In future work we will show that Corollary~\ref{cor:3} is ``tight'' in the sense that for each diagram $D=(\Gamma,J,\pi)$ listed in Tables~\ref{table:1}--\ref{table:5}, every thick irreducible split spherical building~$\Delta$ of type~$\Gamma$ admits an automorphism with diagram $D$. Thus our list of diagrams has no redundancies. There is a minor caveat here: split $\sF_4$ buildings admit dualities if and only if the underlying field has characteristic~$2$, and thus the duality opposition diagram for $\sF_4$ listed in Table~\ref{table:5} is only achieved in characteristic~$2$ (and a similar comment applies for dualities of split $\sB_2$, $\sC_2$ and $\sG_2$ buildings). Moreover we note that if $\Delta$ is not split then it is possible that some of the diagrams in the tables are unobtainable as the opposition diagram of an automorphism of~$\Delta$ (this will be discussed further in future work).

We conclude this introduction with an outline of the structure of the paper. In Section~\ref{sec:1} we provide background and expand on the definitions given above. We also outline the residue techniques that will be used extensively throughout the paper. In Section~\ref{sec:2} we define a class of ``admissible diagrams'' via combinatorial axioms reminiscent of Tits indices. We show that the opposition diagram of every capped automorphism is an admissible diagram, and thus we obtain all possible opposition diagrams for capped automorphisms, proving Theorem~\ref{thm:main2}. We also provide applications to the calculation of the displacement of an automorphism, proving Theorem~\ref{thm:main3}.  

The proof of Theorem~\ref{thm:main} is divided across Sections~\ref{sec:3} and~\ref{sec:4}. In Section~\ref{sec:3} we prove Theorem~\ref{thm:main} for large classical buildings (types $\sA$, $\sB$, $\sC$, and $\sD$). Most of the work here involves a series of lemmas concerning polar spaces, which forms a natural extension and completion of the analysis in~\cite{TTM:12}. In Section~\ref{sec:4} we prove Theorem~\ref{thm:main} for large exceptional buildings (types $\sE$ and $\sF$). It turns out that our residue arguments cover most cases here, with the exception of two particular configurations in $\sF_4$ and $\sE_7$ buildings. In fact a large part of Section~\ref{sec:4} is devoted to showing that the diagram 
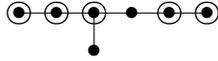
\begin{center}
\begin{tikzpicture}[scale=0.5]
\node [inner sep=0.8pt,outer sep=0.8pt] at (-2,0) (1) {$\bullet$};
\node [inner sep=0.8pt,outer sep=0.8pt] at (-1,0) (3) {$\bullet$};
\node [inner sep=0.8pt,outer sep=0.8pt] at (0,0) (4) {$\bullet$};
\node [inner sep=0.8pt,outer sep=0.8pt] at (1,0) (5) {$\bullet$};
\node [inner sep=0.8pt,outer sep=0.8pt] at (2,0) (6) {$\bullet$};
\node [inner sep=0.8pt,outer sep=0.8pt] at (3,0) (7) {$\bullet$};
\node [inner sep=0.8pt,outer sep=0.8pt] at (0,-1) (2) {$\bullet$};
\draw (-2,0)--(3,0);
\draw (0,0)--(0,-1);
\draw [line width=0.5pt,line cap=round,rounded corners] (1.north west)  rectangle (1.south east);
\draw [line width=0.5pt,line cap=round,rounded corners] (3.north west)  rectangle (3.south east);
\draw [line width=0.5pt,line cap=round,rounded corners] (4.north west)  rectangle (4.south east);
\draw [line width=0.5pt,line cap=round,rounded corners] (6.north west)  rectangle (6.south east);
\draw [line width=0.5pt,line cap=round,rounded corners] (7.north west)  rectangle (7.south east);
\end{tikzpicture}
\end{center}
does not arise as an opposition diagram of any automorphism of a thick $\sE_7$ building. 

\section{Background and definitions}\label{sec:1}

Let $(W,S)$ be a spherical Coxeter system with length function $\ell(\cdot)$ and Coxeter graph $\Gamma=\Gamma(W,S)$. We will adopt Bourbaki~\cite{Bou:02} conventions for the indexing of the generators of irreducible crystallographic spherical systems. For $J\subseteq S$ let $W_J$ be the parabolic subgroup generated by $J$, and let $\Gamma_J=\Gamma(W_J,J)$ be the Coxeter graph of $(W_J,J)$, a subgraph of $\Gamma$. Let $w_0$ be the longest element of $W$, and for each $J\subseteq S$ let  $w_J$ be the longest element of $W_J$. Thus $w_0=w_S$. For each $J\subseteq S$ the element $w_J$  induces a diagram automorphism of $\Gamma_J$, also denoted $w_J$, by $s^{w_J}=w_J^{-1}sw_J$. Note that for irreducible spherical Coxeter systems the automorphism $w_0$ of $\Gamma$ is the identity except for the cases $\sA_n$ with $n\geq 2$, $\sD_n$ with $n$ odd, and $\sE_6$, and in these cases $w_0$ is the unique order $2$ diagram automorphism. We also note that if the Coxeter graph of $(W,S)$ is disconnected then $w_0$ is the product of the longest elements of each connected component, and thus opposition is simply opposition on each component.

Let $\Delta$ be a building of type $(W,S)$. Our main references for the theory of buildings are \cite{AB:08,Tit:74}, and we assume that the reader is already acquainted with the theory. Typically we will regard spherical buildings as simplicial complexes, however at times an incidence geometry approach is more appropriate. Let $\cC=\cC(\Delta)$ denote the set of all chambers (maximal simplices) of $\Delta$, and let $\type:\Delta\to 2^S$ be a fixed type map on the simplicial complex $\Delta$. 

A \textit{panel} of $\Delta$ is a simplex of the form $C\backslash\{v\}$ for some chamber $C$ and some vertex $v$ of $C$. The type of the panel $C\backslash\{v\}$ is $S\backslash\{s\}$ where $s=\tau(v)$, and we call such a panel an \textit{$s$-panel}. Chambers $C,D$ are called \textit{$s$-adjacent} if $C\cap D$ is an $s$-panel. Let $\delta:\cC\times \cC\to W$ be the Weyl distance function associated to the $s$-adjacency relations.

Chambers $C$ and $D$ are \textit{opposite} if and only if they are at maximal distance in the chamber graph (with adjacency given by the union of the $s$-adjacency relations). Equivalently, chambers $C,D\in\cC$ are opposite if and only if
$
\delta(C,D)=w_0
$. The definition of opposition extends to arbitrary simplices as follows.

\begin{defn} Simplices $\alpha,\beta$ of $\Delta$ are \textit{opposite} if for each chamber $A$ containing $\alpha$,  there exists a chamber $B$ containing $\beta$ such that $A$ and $B$ are opposite, and conversely for each chamber $B'$ containing $\beta$ there exists a chamber $A'$ containing $\alpha$ such that $A'$ and $B'$ are opposite. 
\end{defn}

If $J\subseteq S$ we write $J^{\mathrm{op}}=J^{w_0}=w_0^{-1}Jw_0$ (the `opposite type' to $J$). If $\alpha$ and $\beta$ are opposite simplices then necessarily $\tau(\beta)=\tau(\alpha)^{\mathrm{op}}$. Moreover, if $\alpha$ and $\beta$ are simplices with $\tau(\beta)=\tau(\alpha)^{\mathrm{op}}$ then $\alpha$ and $\beta$ are opposite if and only if there exists a chamber $A$ containing $\alpha$ and a chamber $B$ containing $\beta$ such that $A$ and $B$ are opposite.

An \textit{automorphism} of $\Delta$ is a simplicial complex automorphism $\theta:\Delta\to\Delta$. Note that $\theta$ does not necessarily preserve types. Indeed each automorphism $\theta:\Delta\to\Delta$ induces a permutation $\pi_{\theta}$ of the type set $S$, given by $\delta(C,D)=s$ if and only if $\delta(C^{\theta},D^{\theta})=s^{\pi_{\theta}}$, and this permutation $\pi_{\theta}$ is a diagram automorphism of the Coxeter graph~$\Gamma$. When there is no risk of confusion we will often use the symbol $\theta$ for both the automorphism of $\Delta$ and the induced diagram automorphism~$\pi_{\theta}$. If $\Delta$ is irreducible, then from the classification of irreducible spherical Coxeter systems we see that $\theta:S\to S$ is either:

\begin{compactenum}[$(1)$]
\item the identity, in which case $\theta$ is called a \textit{collineation} (or \textit{type preserving}),
\item has order $2$, in which case $\theta$ is called a \textit{duality}, or
\item has order $3$, in which case $\theta$ is called a \textit{triality}; this case only occurs for type $\sD_4$.
\end{compactenum}
Automorphisms $\theta:\Delta\to\Delta$ that induce opposition on the type set are called \textit{oppomorphisms}. For example, oppomorphisms of an $\sE_6$ building are dualities, and oppomorphisms of an $\sE_7$ building are collineations. 

Let $\theta$ be an automorphism of $\Delta$. The \textit{opposite geometry} of $\theta$ is 
$$
\mathrm{Opp}(\theta)=\{\sigma\in\Delta\mid \sigma\text{ is opposite }\sigma^{\theta}\},
$$ 
and the fundamental result of Leeb~\cite{Lee:00} and Abramenko and Brown~\cite{AB:09} mentioned above is as follows (see \cite{DPM:13} for the generalisation to twin buildings).
\begin{thm}[\cite{AB:09,Lee:00}]\label{thm:fund}
If $\theta$ is a nontrivial automorphism of a thick spherical building then $\mathrm{Opp}(\theta)$ is nonempty.
\end{thm}

The following basic properties of $\Opp(\theta)$ are elementary. Note that while we typically use the notation $x^{gh}=(x^g)^h$ for group actions, it is sometimes convenient to regard $gh$ as a function $X\to X$, in which case we write $h\circ g$.

\begin{lemma}\label{lem:opbasic}
Let $\theta$ be an automorphism of a spherical building~$\Delta$. Let $\sigma\in\Opp(\theta)$ and $J=\tau(\sigma)$.
\begin{compactenum}[$(1)$]
\item The set $J$ is stable under the diagram automorphism $w_0\circ\theta$.
\item If $J'\subseteq J$ is stable under $w_0\circ \theta$ and $\sigma'$ is the type $J'$ subsimplex of $\sigma$ then $\sigma'\in\Opp(\theta)$. 
\end{compactenum}
\end{lemma}

\begin{proof}
Since $\sigma$ and $\sigma^{\theta}$ are opposite we have $\tau(\sigma)^{\theta}=\tau(\sigma^{\theta})=\tau(\sigma)^{w_0}$. Thus $J^{\theta}=J^{w_0}$ and since $w_0^2=1$ we have $J^{\theta w_0}=J$, hence~(1). If $J'\subseteq J$ is stable under $w_0\circ\theta$ then $J'^{\theta}=J'^{\op}$. Since $\sigma$ and $\sigma^{\theta}$ are opposite there is a chamber $A$ containing $\sigma$ and a chamber $B$ containing $\sigma^{\theta}$ with $A$ and $B$ opposite. Since $A$ contains $\sigma'$ and $B$ contains $\sigma'^{\theta}$ the simplices $\sigma'$ and $\sigma'^{\theta}$ are opposite. 
\end{proof}

The \textit{type} $\Type(\theta)$ of an automorphism $\theta$ is the union of all subsets $J\subseteq S$ such that there exists a type $J$ simplex mapped to an opposite simplex by~$\theta$. The \textit{opposition diagram} of $\theta$ is the triple $(\Gamma,\Type(\theta),\theta)$, where the third component is the automorphism of $\Gamma$ induced by~$\theta$.

 Less formally, the opposition diagram of $\theta$ is depicted by drawing $\Gamma$ and encircling the nodes of $\Type(\theta)$, where we encircle nodes in minimal subsets invariant under $w_0\circ \theta$. We draw the diagram `bent' (in the standard way) if $w_0\circ\theta\neq 1$. For example, consider the diagrams
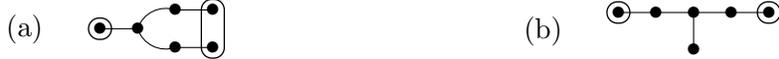
\begin{center}
\begin{tikzpicture}[scale=0.5]
\node at (-4,0) {(a)};
\node at (0,0.3) {};
\node [inner sep=0.8pt,outer sep=0.8pt] at (-2,0) (2) {$\bullet$};
\node [inner sep=0.8pt,outer sep=0.8pt] at (-1,0) (4) {$\bullet$};
\node [inner sep=0.8pt,outer sep=0.8pt] at (0,-0.5) (5) {$\bullet$};
\node [inner sep=0.8pt,outer sep=0.8pt] at (0,0.5) (3) {$\bullet$};
\node [inner sep=0.8pt,outer sep=0.8pt] at (1,-0.5) (6) {$\bullet$};
\node [inner sep=0.8pt,outer sep=0.8pt] at (1,0.5) (1) {$\bullet$};
\draw (-2,0)--(-1,0);
\draw (-1,0) to [bend left=45] (0,0.5);
\draw (-1,0) to [bend right=45] (0,-0.5);
\draw (0,0.5)--(1,0.5);
\draw (0,-0.5)--(1,-0.5);
\draw [line width=0.5pt,line cap=round,rounded corners] (2.north west)  rectangle (2.south east);
\draw [line width=0.5pt,line cap=round,rounded corners] (1.north west)  rectangle (6.south east);
\end{tikzpicture}\qquad\qquad\qquad\qquad\qquad
\begin{tikzpicture}[scale=0.5]
\node at (-4,-0.5) {(b)};
\node at (0,0.3) {};
\node [inner sep=0.8pt,outer sep=0.8pt] at (-2,0) (1) {$\bullet$};
\node [inner sep=0.8pt,outer sep=0.8pt] at (-1,0) (3) {$\bullet$};
\node [inner sep=0.8pt,outer sep=0.8pt] at (0,0) (4) {$\bullet$};
\node [inner sep=0.8pt,outer sep=0.8pt] at (1,0) (5) {$\bullet$};
\node [inner sep=0.8pt,outer sep=0.8pt] at (2,0) (6) {$\bullet$};
\node [inner sep=0.8pt,outer sep=0.8pt] at (0,-1) (2) {$\bullet$};
\draw (-2,0)--(2,0);
\draw (0,0)--(0,-1);
\draw [line width=0.5pt,line cap=round,rounded corners] (1.north west)  rectangle (1.south east);
\draw [line width=0.5pt,line cap=round,rounded corners] (6.north west)  rectangle (6.south east);
\end{tikzpicture} 
\end{center}
Diagram (a) represents a collineation $\theta$ of an $\sE_6$ building with $\Type(\theta)=\{1,2,6\}$, and diagram (b) represents a duality $\theta$ of an $\sE_6$ building with $\Type(\theta)=\{1,6\}$.

We call an opposition diagram \textit{empty} if no nodes are encircled (that is, $\Type(\theta)=\emptyset$), and \textit{full} if all nodes are encircled (that is, $\Type(\theta)=S$). The main concepts of this paper are introduced in the following definition.

\begin{defn}
Let $\Delta$ be a spherical building of type $(W,S)$. Let $\theta$ be a nontrivial automorphism of $\Delta$, and let $J\subseteq S$. Then $\theta$ is called:
\begin{compactenum}[$(1)$]
\item \textit{capped} if there exists a type $\Type(\theta)$ simplex in $\Opp(\theta)$, and \textit{uncapped} otherwise.
\item \textit{domestic} if $\Opp(\theta)$ contains no chamber. 
\item \textit{$J$-domestic} if $\Opp(\theta)$ contains no type $J$ simplex. 
\item \textit{exceptional domestic} if $\theta$ is domestic with full opposition diagram.
\item \textit{strongly exceptional domestic} if $\theta$ is domestic, but not $J$-domestic for any strict subset $J$ of $S$ invariant under $w_0\circ\theta$. 
\end{compactenum}
Note that if $\theta$ is a domestic automorphism with $w_0\circ\theta=1$ then $\theta$ is exceptional domestic if and only if there exists a vertex of each type mapped to an opposite vertex, and $\theta$ is strongly exceptional domestic if and only if there exists a panel of each cotype mapped to an opposite panel (recall that a \textit{panel} is a codimension~$1$ simplex). 
\end{defn}

If $J$ is not stable under the diagram automorphism $w_0\circ\theta$ then by Lemma~\ref{lem:opbasic} the automorphism $\theta$ is necessarily $J$-domestic. For example, if $\theta$ is a nontrivial collineation of a projective plane then $\Opp(\theta)$ necessarily contains neither points nor lines, and hence $\theta$ is both $\{1\}$-domestic and $\{2\}$-domestic. However by Theorem~\ref{thm:fund} $\Opp(\theta)$ is nonempty, and therefore must contain a chamber, and hence $\theta$ is not domestic. For this reason we will reserve the expression ``$J$-domestic'' for subsets $J$ stable under $w_0\circ\theta$, and with this assumption it is true that $J$-domesticity implies domesticity:

\begin{lemma}\label{lem:Jdom}
If $J\subseteq S$ is invariant under $w_0\circ\theta$ and if $\theta$ is $J$-domestic, then $\theta$ is domestic. 
\end{lemma}

\begin{proof}
If $\theta$ is not domestic then there is a chamber $C$ in $\Opp(\theta)$, and by Lemma~\ref{lem:opbasic} the type $J$-simplex $\sigma$ of $C$ is in $\Opp(\theta)$, and hence $\theta$ is not $J$-domestic. 
\end{proof}

\noindent\begin{minipage}{0.6\textwidth}
\begin{example}\label{ex:fano}
The simplest example of an uncapped automorphism is as follows. Consider the Fano plane, drawn as an incidence geometry (with type $1$ vertices represented as ``points'' and type $2$ vertices represented as ``lines''). Let $\theta$ be the duality given by the permutation $\theta=(1,2,3,4,5,6,7,8)(9,10)(11,12,13,14)$, where the points and lines are numbered as in the diagram. The points $11$ and $13$ are the only points mapped to opposite lines, and the lines $12$ and $14$ are the only lines mapped to opposite points. There is no chamber mapped to an opposite, and so $\theta$ is uncapped. 
\end{example}
\end{minipage}
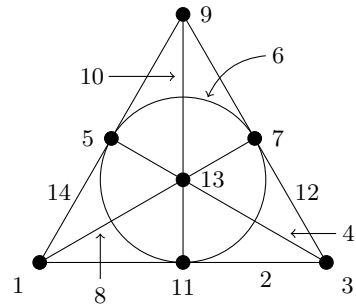
\begin{minipage}{0.4\textwidth}
\begin{center}
\begin{tikzpicture} [scale=1.1]
\path 
(0,0) node (0) [shape=circle,draw,fill=black,scale=0.5,label=right:\footnotesize{$13$}] {}
({cos(30)},{sin(30)}) node (1) [shape=circle,draw,fill=black,scale=0.5,label=right:\footnotesize{$7$}] {}
({cos(150)},{sin(150)}) node (2) [shape=circle,draw,fill=black,scale=0.5,label=left:\footnotesize{$5$}] {}
({cos(270)},{sin(270)}) node (3) [shape=circle,draw,fill=black,scale=0.5,label=below:\footnotesize{$11$}] {}
(0,{1/(sin(30))}) node (4) [shape=circle,draw,fill=black,scale=0.5,label=right:\footnotesize{$9$}] {}
({-1/(tan(30))},{-1}) node (5) [shape=circle,draw,fill=black,scale=0.5,label=below left:\footnotesize{$1$}] {}
({1/(tan(30))},{-1}) node (6) [shape=circle,draw,fill=black,scale=0.5,label=below right:\footnotesize{$3$}] {}
;
\draw (1)--(4)--(2)--(5)--(3)--(6)--(1);
\draw (4)--(0)--(3);
\draw (5)--(0)--(1);
\draw (6)--(0)--(2);
\draw (0,0) circle (1cm);
\path (1.5,-0.15) node {\footnotesize{$12$}};
\path (-1.5,-0.15) node {\footnotesize{$14$}};
\draw [->] (-0.9,1.25)--(-0.1,1.25);
\path (-1.1,1.25) node {\footnotesize{$10$}};
\draw [bend left,<-] (0.3,1) to (1,1.5);
\path (1.15,1.5) node {\footnotesize{$6$}};
\draw [<-] (-1,-0.65)--(-1,-1.2);
\path (-1,-1.4) node {\footnotesize{$8$}};
\draw [<-] (1.3,-0.65)--(1.85,-0.65);
\path (2,-0.65) node {\footnotesize{$4$}};
\path (1,-1.2) node {\footnotesize{$2$}};
\end{tikzpicture}
\end{center}
\end{minipage}
\medskip

Note that the above example is also exceptional domestic and strongly exceptional domestic. In fact, the concepts of uncapped, exceptional domestic, and strongly exceptional domestic are clearly equivalent for rank~$2$ buildings. More generally, the connections between these concepts are given in the following proposition. 

\begin{prop}
Let $\theta$ be a nontrivial automorphism of a spherical building $\Delta$ of type $(W,S)$. 
\begin{compactenum}[$(1)$]
\item If $\theta$ is strongly exceptional domestic then $\theta$ is exceptional domestic. 
\item If $\theta$ is exceptional domestic then $\theta$ is uncapped. 
\item If $\theta$ is uncapped then $\theta$ is domestic. 
\end{compactenum}
\end{prop}

\begin{proof}
(1) Let $\theta$ be strongly exceptional domestic. Note that $\emptyset\neq \Type(\theta)\subseteq S$ is stable under $w_0\circ\theta$ (since it is a union of stable subsets), and hence $J=S\backslash\Type(\theta)$ is stable under $w_0\circ\theta$. Thus, since $J\neq S$, $J=\emptyset$ by the definition of strongly exceptional domestic, and hence $\Type(\theta)=S$, and so $\theta$ is exceptional domestic. 

(2) Let $\theta$ be exceptional domestic. Then $\Type(\theta)=S$, yet since $\theta$ is domestic there is no type $S=\Type(\theta)$ simplex (ie, no chamber) in $\Opp(\theta)$, and so $\theta$ is uncapped. 

(3) Let $\theta$ be uncapped. Hence there is no type $\Type(\theta)$ simplex in $\Opp(\theta)$. Thus $\theta$ is $J$-domestic for the $w_0\circ\theta$ stable set $J=\Type(\theta)$, and hence $\theta$ is domestic by Lemma~\ref{lem:Jdom}. 
\end{proof}

Before continuing we make three remarks concerning the above definitions.

\begin{remark}
In general the above concepts are distinct. It is easy to find examples of domestic automorphisms that are capped. For example symplectic polarities of projective spaces (see Lemma~\ref{lem:sp}), or central collineations of generalised quadrangles. It is, however, harder to find examples of (a) uncapped automorphisms that are not exceptional domestic, and (b) exceptional domestic automorphisms that are not strongly exceptional domestic. Examples will be provided in future work on uncapped automorphisms of small buildings, and so for now we simply state that the smallest building admitting an example of (a) is the polar space $\sC_4(2)$ of $\mathsf{Sp}_8(2)$, and the smallest building admitting an example of (b) is the polar space $\sC_3(2)$ of $\mathsf{Sp}_6(2)$.
\end{remark}

\begin{remark}
A group theoretic interpretation of domesticity is as follows. Let $\Delta$ be a thick irreducible spherical building of rank at least $3$. Thus by Tits' classification~\cite{Tit:74} the type preserving automorphism group $G=\mathrm{Aut}(\Delta)$ admits a $BN$-pair, where $B$ is the stabiliser of a chamber $C_0$, and $N$ is the normaliser of an apartment containing $C_0$. Since $\delta(hB,ghB)=w$ if and only if $h^{-1}gh\in BwB$ we see that an automorphism $g\in G$ is domestic if and only if $g$ is not conjugate to any element of $Bw_0B$. In particular, since $Bw_0B$ is the ``largest'' double coset in the Bruhat decomposition $G=\bigsqcup_{w\in W}BwB$ the above interpretation shows, in a loose sense, that domestic automorphisms are relatively ``rare''. 
\end{remark}

\begin{remark}
Call an automorphism $\theta$ ``strongly capped'' if each simplex $\sigma\in\Opp(\theta)$ is contained in a type $\Type(\theta)$ simplex of $\Opp(\theta)$. This is obviously a stronger condition than cappedness. We note that while there are many examples of automorphisms with this strongly capped property (for example, symplectic polarities of projective spaces), there are also many examples of automorphisms that are not strongly capped. For example, consider the projective space $\sA_3(\FF)$ where $\FF$ is any field, and let $\theta$ be the duality induced by the linear map $g:e_1\mapsto e_1-e_2$, $e_2\mapsto -e_1$, $e_3\mapsto -e_4$, $e_4\mapsto e_3$ (that is, $V^{\theta}=(gV)^{\perp}$). Then $\theta$ is not domestic, for example the chamber $\langle e_1\rangle\subseteq \langle e_1,e_2+e_3\rangle\subseteq \langle e_1,e_2+e_3,e_4\rangle$ is mapped onto an opposite chamber. The line $L=\langle e_3,e_4\rangle$ is also mapped onto an opposite line, however no point on $L$ nor plane through $L$ is mapped onto an opposite. The prevalence of such counter examples to ``strong cappedness'' has lead us to believe that our ``weaker'' notion of cappedness is the appropriate one. 
\end{remark}




Residue arguments are used extensively in our proofs, and so we conclude this section with a summary of the main techniques, along with an example of how they are applied. We first briefly define residues and projections (see \cite{AB:08,Tit:74} for details). The \textit{residue} $\Res(\sigma)$ of a simplex $\sigma\in\Delta$ is the set of all simplices of $\Delta$ which contain~$\sigma$, together with the order relation induced by that of~$\Delta$. Then $\Res(\sigma)$ is a building whose diagram is obtained from the diagram of $\Delta$ by removing all nodes which belong to $\tau(\sigma)$. 

Let $\alpha$ be a simplex of $\Delta$. The \textit{projection onto $\alpha$} is the map $\proj_{\alpha}:\Delta\to\Res(\alpha)$ defined as follows (see \cite[Section~3]{Tit:74}). Firstly, if $B$ is a chamber of $\Delta$ then there is a unique chamber $A\in \Res(\alpha)$ such that $\ell(\delta(A,B))<\ell(\delta(A',B))$ for all chambers $A'\in \Res(\alpha)$ with $A'\neq A$, and we define $\proj_{\alpha}(B)=A$. In other words, $\proj_{\alpha}(B)$ is the unique chamber $A$ of $\Res(\alpha)$ with the property that every minimal length gallery from $B$ to $\Res(\alpha)$ ends with the chamber~$A$. Now, if $\beta$ is an arbitrary simplex we define
$$
\proj_{\alpha}(\beta)=\bigcap_{B}\,\proj_{\alpha}(B)
$$
where the intersection is over all chambers $B$ in $\Res(\beta)$. In other words, $\proj_{\alpha}(\beta)$ is the unique simplex $\gamma$ of $\Res(\alpha)$ which is maximal subject to the property that every minimal length gallery from a chamber of $\Res(\beta)$ to $\Res(\alpha)$ ends in a chamber containing~$\gamma$.

Let $\theta$ be an automorphism of $\Delta$, and suppose that $\sigma\in\Opp(\theta)$. It follows from \cite[Theorem~3.28]{Tit:74} that the projection map $\proj_{\sigma}:\Res(\sigma^{\theta})\to\Res(\sigma)$ is an isomorphism. Define
$$
\theta_{\sigma}:\Res(\sigma)\xrightarrow{\sim} \Res(\sigma)\quad \text{by}\quad \theta_{\sigma}=\proj_{\sigma}\circ\,\theta.
$$
The type map induced by $\theta_{\sigma}$ is as follows.

\begin{prop}\label{prop:typemap}
Let $\theta$ be an automorphism of a spherical building $\Delta$ of type $(W,S)$. Suppose that $\sigma\in\Opp(\theta)$ and let $J=\tau(\sigma)$. Then the type map on $S\backslash J$ induced by $\theta_{\sigma}$ is $w_{S\backslash J}\circ w_0\circ \theta$. 
\end{prop}

\begin{proof}
This follows easily from \cite[Corollary~5.116]{AB:08}.
\end{proof}

\begin{example}\label{ex:D}
In words, Proposition~\ref{prop:typemap} says that to compute the induced action on the types of $S\backslash J$, compute the type map of $\theta$, followed by opposition of the graph $\Gamma$, followed by opposition of the graph $\Gamma_{S\backslash J}$. We will use this proposition many times, often without reference. For example, consider a duality $\theta$ of a $\sD_n$ building, and suppose that $v\in\Opp(\theta)$ is a type $i$ vertex, with $i\leq n-2$. The residue of $v$ is a building of type $\sA_{i-1}\times \sD_{n-i}$, and the induced automorphism $\theta_v$ of $\Res(v)$ is a duality on the $\sA_{i-1}$ component (with the convention that duality~$\equiv$~collineation for~$\sA_1$), and a duality (respectively collineation) on the $\sD_{n-i}$ component if~$i$ is even (respectively odd). 
\end{example}

The following proposition is used repeatedly throughout this paper.
 
\begin{prop}{\cite[Proposition~3.29]{Tit:74}}\label{prop:proj} Let $\theta$ be an automorphism of a spherical building $\Delta$ and let $\alpha\in\Opp(\theta)$. If $\beta\in\Res(\alpha)$ then $\beta$ is opposite $\beta^{\theta}$ in the building $\Delta$ if and only if $\beta$ is opposite $\beta^{\theta_{\alpha}}$ in the building $\Res(\alpha)$. 
\end{prop}

The following corollary facilitates inductive residue arguments. 

\begin{cor}\label{cor:proj}
Let $\theta:\Delta\to\Delta$ be a domestic automorphism and let $\sigma\in\Opp(\theta)$. Then $\theta_{\sigma}:\Res(\sigma)\to \Res(\sigma)$ is a domestic automorphism of the building~$\Res(\sigma)$.
\end{cor}

\begin{proof}
Let $J=\tau(\sigma)$. If $\theta_{\sigma}$ is not domestic then there is a chamber $\sigma'$ of $\Res(\sigma)$ mapped onto an opposite chamber by $\theta_{\sigma}$. Then $\sigma\cup\sigma'$ is a chamber of $\Delta$, and from Proposition~\ref{prop:proj} this chamber is mapped onto an opposite chamber, a contradiction. 
\end{proof}

In light of Corollary~\ref{cor:proj} it is natural that the theory of domesticity in rank~$2$ buildings (ie, generalised polygons) plays a central role. This analysis has been undertaken in \cite{PTM:15,TTM:09,TTM:12b}, and since only residues of types $\sA_1\times\sA_1$, $\sA_2$, and $\sB_2/\sC_2$ appear as rank $2$ residues of irreducible thick spherical buildings of rank $3$ or more, the relevant results are as follows.

\begin{thm}{\cite{PTM:15,TTM:09,TTM:12b}}\label{thm:rank2} Let $\Delta$ be a thick generalised $n$-gon. If $n$ is odd then no nontrivial collineation is domestic, and if $n$ is even then no duality is domestic. Moreover, 
\begin{compactenum}[$(1)$]
\item If $n=2$ then $\Delta$ admits no exceptional domestic automorphisms. 
\item If $n=3$ then $\Delta$ admits an exceptional domestic duality if and only if $\Delta$ is a Fano plane, and in this case there exists a unique exceptional domestic duality up to conjugation.
\item If $n=4$ then $\Delta$ admits an exceptional domestic collineation if and only if the generalised quadrangle $\Delta$ has parameters $(2,2)$, $(2,4)$, $(4,2)$, $(3,5)$, or $(5,3)$, and in each case there exists a unique exceptional domestic collineation up to conjugation.
\end{compactenum}
\end{thm}

Recall that for generalised polygons an automorphism is uncapped if and only if it is exceptional domestic. Thus the above theorem shows that uncapped automorphisms of rank~$2$ buildings are very rare, and are restricted to ``small'' buildings. Theorem~\ref{thm:main} shows that this pattern continues into higher rank spherical buildings.

\begin{example}\label{ex:E7}
We now provide a detailed example of our residue arguments by outlining how  Propositions~\ref{prop:typemap} and \ref{prop:proj} can be used to prove Theorem~\ref{thm:main} in the case of $\sE_7$ buildings. Similar residue techniques will be used repeatedly throughout the paper, often with briefer explanations. Recall that we adopt the Bourbaki labelling conventions:
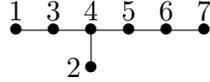
\begin{center}
\begin{tikzpicture}[scale=0.5]
\node [inner sep=0.8pt,outer sep=0.8pt] at (-2,0) (1) {$\bullet$};
\node [inner sep=0.8pt,outer sep=0.8pt] at (-1,0) (3) {$\bullet$};
\node [inner sep=0.8pt,outer sep=0.8pt] at (0,0) (4) {$\bullet$};
\node [inner sep=0.8pt,outer sep=0.8pt] at (1,0) (5) {$\bullet$};
\node [inner sep=0.8pt,outer sep=0.8pt] at (2,0) (6) {$\bullet$};
\node [inner sep=0.8pt,outer sep=0.8pt] at (3,0) (7) {$\bullet$};
\node [inner sep=0.8pt,outer sep=0.8pt] at (0,-1) (2) {$\bullet$};
\draw (-2,0)--(3,0);
\draw (0,0)--(0,-1);
\node [above] at (1) {$1$};
\node [left] at (2) {$2$};
\node [above] at (3) {$3$};
\node [above] at (4) {$4$};
\node [above] at (5) {$5$};
\node [above] at (6) {$6$};
\node [above] at (7) {$7$};
\end{tikzpicture}
\end{center}
We will require the following five facts that will be proved later in the paper. The first four facts are relatively easy, and their utility arises naturally due the presence of $\sA_n$, $\sD_n$, and $\sE_6$ residues in an $\sE_7$ building. The fifth fact is of a rather different flavour, and turns out to be highly nontrivial. 
\begin{compactenum}
\item[(1)] No duality of a large $\sA_{2n}$ building is domestic (Theorem~\ref{thm:Alarge}).
\item[(2)] No duality of a large $\sA_{2n+1}$ building is $\{2,4,\ldots,2n\}$-domestic (Theorem~\ref{thm:Alarge} and Lemma~\ref{lem:sp}).
\item[(3)] No duality of a thick $\sD_n$ building is $\{1\}$-domestic (Proposition~\ref{prop:DDual}).
\item[(4)] No duality of a large $\sE_6$ building is $\{1,6\}$-domestic (Lemma~\ref{lem:exceptional} and Theorem~\ref{thm:fund}).
\item[(5)] If $\theta$ is $\{3,7\}$-domestic collineation of a thick $\sE_7$ building then $\theta$ is either $\{3\}$-domestic or $\{7\}$-domestic (Proposition~\ref{prop:summary2}).
\end{compactenum}
Let $\theta$ be a collineation of a large $\sE_7$ building~$\Delta$, and let $J=\Type(\theta)$. We are required to show that there exists a type $J$ simplex in $\Opp(\theta)$, or equivalently that $\theta$ is not $J$-domestic.

We make the following claims:
\begin{compactenum}
\item[(a)] If either $2\in J$ or $5\in J$ then $\theta$ is not domestic.
\item[(b)] If either $3\in J$ or $4\in J$ then $\theta$ is not $\{1,3,4,6\}$-domestic.
\item[(c)] If $6\in J$ then $\theta$ is not $\{1,6\}$-domestic.
\item[(d)] If $7\in J$ then $\theta$ is not $\{1,6,7\}$-domestic.
\end{compactenum}
To prove (a), suppose first that $2\in J$ and let $v$ be a type $2$ vertex in $\Opp(\theta)$. Then the automorphism $\theta_v$ of the large type $\sA_6$ building $\mathrm{Res}(v)$ is a duality (by Proposition~\ref{prop:typemap}), and so by~(1) $\theta_v$ maps a chamber $\sigma$ of $\mathrm{Res}(v)$ to an opposite chamber of $\mathrm{Res}(v)$. Thus by Proposition~\ref{prop:proj} the chamber $\sigma\cup\{v\}$ of $\Delta$ is mapped onto an opposite by $\theta$, and so $\theta$ is not domestic. 

Suppose now that $5\in J$, and let $v$ be a type $5$ vertex of $\Opp(\theta)$. Then the automorphism $\theta_v$ of the type $\sA_4\times\sA_2$ building $\mathrm{Res}(v)$ acts as a duality on each component (by Proposition~\ref{prop:typemap}), and so by~(1) there is a chamber $\sigma_1$ (respectively $\sigma_2$) of the $\sA_4$ (respectively $\sA_2$) component mapped onto an opposite chamber by $\theta_v$. Thus by Proposition~\ref{prop:proj} the chamber $\sigma_1\cup\sigma_2\cup\{v\}$ of $\Delta$ is mapped onto an opposite chamber by~$\theta$, and hence the proof of (a) is complete. 

To prove (b), suppose first that $4\in J$, and let $v$ be a type $4$ vertex of $\Opp(\theta)$. Then $\theta_v$ is an automorphism of a type $\sA_1\times\sA_2\times\sA_3$ building acting as a duality on the $\sA_2$ and $\sA_3$ components (by Proposition~\ref{prop:typemap}), and so by~(1), (2), and Proposition~\ref{prop:proj}, there is a type $\{1,3\}\cup\{6\}\cup\{4\}$ simplex of $\Delta$ mapped onto an opposite simplex. Thus $\theta$ is not $\{1,3,4,6\}$-domestic.

Suppose that $3\in J$, and let $v$ be a type $3$ vertex of $\Opp(\theta)$. Then $\theta_v$ acts as a duality on the $\sA_5$ component of $\mathrm{Res}(v)$ (by Proposition~\ref{prop:typemap}), and so by~(2) $\theta_v$ maps a type $\{2,4\}$ simplex of this residue to an opposite simplex. Thus by Proposition~\ref{prop:proj} we have $4\in J$, and hence by the previous paragraph $\theta$ is not $\{1,3,4,6\}$-domestic, completing the proof of~(b).

To prove (c), suppose that $6\in J$, and let $v$ be a type $6$ vertex of $\Opp(\theta)$. Then $\theta_v$ acts as a duality on the $\sD_5$ component of $\mathrm{Res}(v)$ (by Proposition~\ref{prop:typemap}), and hence by~(1) and Proposition~\ref{prop:proj} we see that $\theta$ is not $\{1,6\}$-domestic.

To prove (d), suppose that $7\in J$, and let $v$ be a type $7$ vertex of $\Opp(\theta)$. Then $\theta_v$ acts as a duality on the type $\sE_6$ building $\mathrm{Res}(v)$ (by Proposition~\ref{prop:typemap}), and hence by~(4) and Proposition~\ref{prop:proj} we see that $\theta$ is not $\{1,6,7\}$-domestic. 

It follows statements (a), (b), (c), and (d) that the possibilities for $J$ are $J=\{1\}$, $\{1,6\}$, $\{1,6,7\}$, $\{1,3,4,6\}$, or $\{1,3,4,6,7\}$, and that in the first four cases $\theta$ is capped. Thus it remains to consider the possitility $J=\{1,3,4,6,7\}$. This case is not approachable by residue arguments alone. However by~(5) there exists a type $\{3,7\}$ simplex $\sigma$ in $\Opp(\theta)$, and by Proposition~\ref{prop:typemap} the automorphism $\theta_{\sigma}$ of the type $\sA_1\times\sA_4$ building $\mathrm{Res}(\sigma)$ acts as a duality on the $\sA_4$ component (by Proposition~\ref{prop:typemap}). Hence by~(1) and Proposition~\ref{prop:proj} there is a type $\{2,4,5,6\}\cup\{3,7\}$ simplex in $\Opp(\theta)$. In particular there is a type $2$ vertex in $\Opp(\theta)$, and so by~(a) $\theta$ is not domestic. Thus $J=\{1,3,4,6,7\}$ is in fact impossible, completing the proof that every collineation of a large $\sE_7$ building is capped.
\end{example}

\section{Opposition diagrams}\label{sec:2}

In this section we introduce a diagram combinatorics for opposition diagrams of capped automorphisms, and use this combinatorial approach to classify the possible opposition diagrams of capped automorphisms, proving Theorem~\ref{thm:main2}. We also give an application to the calculation of displacement of a capped automorphism of a spherical building, proving Theorem~\ref{thm:main3}.

\subsection{Admissible diagrams}

Let $(W,S)$ be a (not necessarily irreducible) spherical Coxeter system with Coxeter graph $\Gamma$, and let $\pi$ be an automorphism of $\Gamma$. If $K\subseteq S$ is invariant under $w_0\circ \pi$ then $S\backslash K$ is also invariant under $w_0\circ \pi$, and hence 
$$
\pi_K=w_{S\backslash K}\circ w_0\circ \pi:S\backslash K\to S\backslash K
$$ is an automorphism of the subgraph $\Gamma_{S\backslash K}$.

Let $\mathcal{A}$ denote the set of all triples $(\Gamma,J,\pi)$ where $\Gamma=\Gamma(W,S)$ is the Coxeter graph of some spherical Coxeter system $(W,S)$, $J\subseteq S$ is a subset of the vertex set of $\Gamma$, and $\pi$ is a graph automorphism of $\Gamma$, such that the following axioms hold:
\begin{compactenum}[$(1)$]
\item if $\pi\neq 1$ then $J\neq \emptyset$,
\item $J$ is closed under the actions of both $w_0$ and $\pi$,
\item if $K\subseteq J$ is closed under $w_0\circ \pi$ then $(\Gamma_{S\backslash K},J\backslash K,\pi_K)\in\mathcal{A}$. 
\end{compactenum}
This is an inductive definition, and thus we need to specify the permitted ``base cases''. These are taken to be $(\bullet,\{\},\mathrm{id})$ and $(\bullet,\{s\},\mathrm{id})$, where $\bullet$ is the Coxeter graph of the type $A_1$ Coxeter system with $S=\{s\}$. 

The elements $(\Gamma,J,\pi)$ of $\cA$ are called \textit{admissible diagrams}. The admissible diagram $(\Gamma,J,\pi)$ will be depicted by drawing the graph $\Gamma$ and encircling each minimal $w_0\circ\pi$ invariant subset of $J$. We draw the diagram `bent' (in the standard way, c.f. Section~\ref{sec:1}) if $w_0\circ\pi\neq 1$. We call a diagram \textit{empty} if no nodes are encircled, and \textit{full} if all nodes are encircled. We note that axiom (3) can be interpreted as saying that the set $\mathcal{A}$ is closed under ``taking residues'' in an appropriate sense.

The connection between admissible diagrams and capped automorphisms is given in the following elementary proposition. 

\begin{prop}\label{prop:conn}
If $\theta$ is a capped automorphism of a thick spherical building of type $(W,S)$ then $(\Gamma,J,\pi)$ is an admissible diagram, where $\Gamma=\Gamma(W,S)$, $J=\Type(\theta)$, and $\pi=\pi_{\theta}$ is the automorphism of $\Gamma$ induced by~$\theta$.
\end{prop}

\begin{proof}
If $\pi\neq\mathrm{id}$ then $\theta$ is a nontrivial automorphism, and hence $\Opp(\theta)\neq\emptyset$ by Theorem~\ref{thm:fund}. Thus $J\neq \emptyset$. 

By cappedness there exists a simplex $\sigma$ of type $J$ mapped to an opposite simplex by $\theta$, and so by Lemma~\ref{lem:opbasic}(1) $J$ is stable under $w_0\circ \pi$. Moreover, since $\sigma$ and $\sigma^{\theta}$ are opposite, $\sigma^{\theta^{-1}}$ and $\sigma$ are also opposite, and hence there is a simplex of type $J^{\pi^{-1}}$ mapped to an opposite by $\theta$. Cappedness forces $J^{\pi^{-1}}=J$, and hence $J$ is stable under $\pi$. Since $J$ is also stable under $w_0\circ \pi$ we see that $J$ is stable under both $w_0$ and $\pi$. 

Suppose that $K\subseteq J$ is stable under $w_0\circ\pi$. Then by Lemma~\ref{lem:opbasic}(2) the type $K$ subsimplex $\sigma'$ of $\sigma$ is mapped to an opposite by $\theta$. The induced automorphism $\theta_{\sigma'}:\mathrm{Res}(\sigma')\to \mathrm{Res}(\sigma')$ is an automorphism of a building of type $(W_{S\backslash K},S\backslash K)$ inducing $\pi_K=w_{S\backslash K}\circ w_0\circ \pi$ on the type set (see Proposition~\ref{prop:typemap}). Moreover, $\theta_{\sigma'}$ is capped by Proposition~\ref{prop:proj}. By the above paragraphs the automorphism $\theta_{\sigma'}$ satisfies conditions (1) and (2), and since $\theta_{\sigma'}$ is capped we can continue in this way until we reach a diagram with $J=\emptyset$ and then necessarily the associated diagram automorphism is trivial. Such a diagram satisfies all three axioms, and hence the result.
\end{proof}

We now classify the connected admissible diagrams. We first note that $\cA$ has the following obvious closure properties. 

\begin{lemma}\label{lem:closed}
Let $(W_1,S_1)$ and $(W_2,S_2)$ be disjoint spherical Coxeter systems (not necessarily irreducible) and let $(W,S)=(W_1\times W_2,S_1\cup S_2)$. Let $\Gamma$ be the Coxeter graph of $(W,S)$, and let $\Gamma_i$ be the Coxeter graph of $(W_i,S_i)$ for $i=1,2$. 
\begin{compactenum}[$(1)$]
\item If $(\Gamma_i,J_i,\pi_i)\in\cA$ for $i=1,2$, then $(\Gamma,J_1\cup J_2,\pi_1\circ\pi_2)\in\cA$. 
\item If $(\Gamma,J,\pi)\in\cA$ and $\pi(S_i)=S_i$ for $i=1,2$ then $(\Gamma_i,J\cap J_i,\pi|_{S_i})\in\cA$ for $i=1,2$.
\end{compactenum}
\end{lemma}

\newpage

\begin{thm}\label{thm:diagrams}
The elements $(\Gamma,J,\pi)$ of $\mathcal{A}$ with $\Gamma$ connected are as listed in Table~\ref{table:1}--Table~\ref{table:5} (with $\pi=\pi_{\theta}$), along with Table~\ref{table:6} below.

\begin{center}
\noindent\begin{tabular}{|l|}
\hline
$\Gamma=\mathsf{H}_3/\mathsf{H}_4$, $\pi=\mathrm{id}$\\
\hline\hline
\begin{tabular}{l}
\begin{tikzpicture}[scale=0.5,baseline=-0.5ex]
\node at (0,0.8) {};
\node at (0,-0.8) {};
\node at (0,0) (0) {};
\node [inner sep=0.8pt,outer sep=0.8pt] at (0,0) (0) {$\bullet$};
\node [inner sep=0.8pt,outer sep=0.8pt] at (1,0) (1) {$\bullet$};
\node [inner sep=0.8pt,outer sep=0.8pt] at (2,0) (2) {$\bullet$};
\draw (0,0)--(2,0);
\node at (1.5,0) [above] {$5$};
\end{tikzpicture}\quad
\begin{tikzpicture}[scale=0.5,baseline=-0.5ex]
\node at (0,0.8) {};
\node at (0,-0.8) {};
\node at (0,0) (0) {};
\node [inner sep=0.8pt,outer sep=0.8pt] at (0,0) (0) {$\bullet$};
\node [inner sep=0.8pt,outer sep=0.8pt] at (1,0) (1) {$\bullet$};
\node [inner sep=0.8pt,outer sep=0.8pt] at (2,0) (2) {$\bullet$};
\draw (0,0)--(2,0);
\node at (1.5,0) [above] {$5$};
\draw [line width=0.5pt,line cap=round,rounded corners] (1.north west)  rectangle (1.south east);
\end{tikzpicture}\quad
\begin{tikzpicture}[scale=0.5,baseline=-0.5ex]
\node at (0,0.8) {};
\node at (0,-0.8) {};
\node at (0,0) (0) {};
\node [inner sep=0.8pt,outer sep=0.8pt] at (0,0) (0) {$\bullet$};
\node [inner sep=0.8pt,outer sep=0.8pt] at (1,0) (1) {$\bullet$};
\node [inner sep=0.8pt,outer sep=0.8pt] at (2,0) (2) {$\bullet$};
\draw (0,0)--(2,0);
\node at (1.5,0) [above] {$5$};
\draw [line width=0.5pt,line cap=round,rounded corners] (0.north west)  rectangle (0.south east);
\draw [line width=0.5pt,line cap=round,rounded corners] (1.north west)  rectangle (1.south east);
\draw [line width=0.5pt,line cap=round,rounded corners] (2.north west)  rectangle (2.south east);
\end{tikzpicture}\quad
\begin{tikzpicture}[scale=0.5,baseline=-0.5ex]
\node at (0,0.8) {};
\node at (0,-0.8) {};
\node at (0,0) (0) {};
\node [inner sep=0.8pt,outer sep=0.8pt] at (-1,0) (-1) {$\bullet$};
\node [inner sep=0.8pt,outer sep=0.8pt] at (0,0) (0) {$\bullet$};
\node [inner sep=0.8pt,outer sep=0.8pt] at (1,0) (1) {$\bullet$};
\node [inner sep=0.8pt,outer sep=0.8pt] at (2,0) (2) {$\bullet$};
\draw (-1,0)--(2,0);
\node at (1.5,0) [above] {$5$};
\end{tikzpicture}\quad
\begin{tikzpicture}[scale=0.5,baseline=-0.5ex]
\node at (0,0.8) {};
\node at (0,-0.8) {};
\node at (0,0) (0) {};
\node [inner sep=0.8pt,outer sep=0.8pt] at (-1,0) (-1) {$\bullet$};
\node [inner sep=0.8pt,outer sep=0.8pt] at (0,0) (0) {$\bullet$};
\node [inner sep=0.8pt,outer sep=0.8pt] at (1,0) (1) {$\bullet$};
\node [inner sep=0.8pt,outer sep=0.8pt] at (2,0) (2) {$\bullet$};
\draw (-1,0)--(2,0);
\node at (1.5,0) [above] {$5$};
\draw [line width=0.5pt,line cap=round,rounded corners] (-1.north west)  rectangle (-1.south east);
\end{tikzpicture}\quad
\begin{tikzpicture}[scale=0.5,baseline=-0.5ex]
\node at (0,0.8) {};
\node at (0,-0.8) {};
\node at (0,0) (0) {};
\node [inner sep=0.8pt,outer sep=0.8pt] at (-1,0) (-1) {$\bullet$};
\node [inner sep=0.8pt,outer sep=0.8pt] at (0,0) (0) {$\bullet$};
\node [inner sep=0.8pt,outer sep=0.8pt] at (1,0) (1) {$\bullet$};
\node [inner sep=0.8pt,outer sep=0.8pt] at (2,0) (2) {$\bullet$};
\draw (-1,0)--(2,0);
\node at (1.5,0) [above] {$5$};
\draw [line width=0.5pt,line cap=round,rounded corners] (-1.north west)  rectangle (-1.south east);
\draw [line width=0.5pt,line cap=round,rounded corners] (0.north west)  rectangle (0.south east);
\draw [line width=0.5pt,line cap=round,rounded corners] (1.north west)  rectangle (1.south east);
\draw [line width=0.5pt,line cap=round,rounded corners] (2.north west)  rectangle (2.south east);
\end{tikzpicture}
\end{tabular}\\
\hline
\end{tabular}
\captionof{table}{Admissible diagrams for noncrystallographic types $\mathsf{H}_3$ and $\mathsf{H}_4$}\label{table:6}
\end{center}
\end{thm}

\begin{proof}
One must show two things: Firstly that each of the listed diagrams are elements of $\cA$, and secondly that the list is complete. The first task is a simple induction on rank, making use of Lemma~\ref{lem:closed}, and we omit the details. Thus we show that the list is complete. 

Consider type $\sA_n$, with $\pi=\id$. Minimal sets invariant under $w_0\circ\pi$ are of the form $\{i,n-i+1\}$ with $1\leq i\leq \lfloor n/2\rfloor$. Suppose that $J\neq \emptyset$, and let $i$ be minimal subject to $\{i,n-i+1\}\subseteq J$. Suppose that $i>1$. The diagram $\Gamma_{S\backslash \{i,n-i+1\}}$ is of type $\sA_{i-1}\times\sA_{i-1}\times \sA_k$ for some $k$, and $\pi_{\{i,n-i+1\}}$ interchanges the two $\sA_{i-1}$ components. Thus the restriction of $\pi_{\{i,n-i+1\}}$ to the $\sA_{i-1}\times\sA_{i-1}$ subdiagram is nontrivial, contradicting Axiom~(1). Thus $i=1$. Let $1\leq j\leq \lfloor n/2\rfloor$ be maximal subject to $\{j,n-j+1\}\subseteq J$. If there is an index $1<k<j$ such that $\{k,n-k+1\}\not\subseteq J$ then the diagram $\Gamma_{S\backslash J}$ contains an empty subdiagram of type $\sA_r\times\sA_r$ with $\pi_J$ interchanging the two components, a contradiction. Hence the result holds in this case.

Consider type $\sA_n$, with $\pi\neq\id$. Thus $n>1$ and $J\neq\emptyset$. For $n=2$ it is clear by Axioms (1) and (2) that the only possible admissible diagram is \begin{tikzpicture}[scale=0.5,baseline=-0.5ex]
\node at (1,0.3) {};
\node at (1,-0.3) {};
\node [inner sep=0.8pt,outer sep=0.8pt] at (1,0) (1) {$\bullet$};
\node [inner sep=0.8pt,outer sep=0.8pt] at (2,0) (2) {$\bullet$};
\draw (1,0)--(2,0);
\draw [line width=0.5pt,line cap=round,rounded corners] (1.north west)  rectangle (1.south east);
\draw [line width=0.5pt,line cap=round,rounded corners] (2.north west)  rectangle (2.south east);
\end{tikzpicture}, starting an induction. The diagram $\Gamma_{S\backslash J}$ is empty of type $\sA_{n_1}\times\cdots\times \sA_{n_k}$, and $\pi_J=w_{S\backslash J}$ acts on each $\sA_{n_i}$ by opposition. Hence Axiom (1) forces $n_i=1$ for all $i=1,\ldots,k$. Thus if $j$ is the maximal element of $J$ we have $j=n-1$ or $j=n$. Suppose that $n$ is odd. If $j=n$ then $\Gamma_{S\backslash\{j\}}$ has type $\sA_{n-1}$ with $n-1$ even, and by the induction hypothesis $J=\{1,\ldots,n-1\}\cup\{n\}$. If $j=n-1$ then by $\Gamma_{S\backslash \{j\}}$ has type $\sA_{n-2}\times \sA_1$ and $\pi_{\{j\}}$ acts as opposition on the $\sA_{n-2}$ component. Thus by Lemma~\ref{lem:closed} and the induction hypothesis either $J=\{1,2,\ldots,n-2\}\cup\{n-1\}$ or $J=\{2,4,\ldots,n-3\}\cup \{n-1\}$. The former case is impossible, because by the induction hypothesis $(\Gamma_{S\backslash\{1\}},J\backslash\{1\},\pi_{\{1\}})$ is not in $\cA$. Now suppose that $n$ is even. If $j=n$ then by the induction hypothesis $J=\{1,2,\ldots,n-1\}\cup\{n\}$ or $J=\{2,4,\ldots,n-2\}\cup\{n\}$. The latter case is impossible, since $(\Gamma_{\{2\}},J\backslash\{2\},\pi_{\{2\}})\notin\cA$. If $j=n-1$ then by the induction hypothesis $J=\{1,2,\ldots,n-2\}\cup\{n-1\}$, a contradiction since $(\Gamma_{S\backslash\{1\}},J\backslash\{1\},\pi_{\{1\}})\notin\cA$.

Consider type $\sB_n=\sC_n$, with $\pi=\mathrm{id}$. If $n=2$ then the possible admissible diagrams are
 \begin{tikzpicture}[scale=0.5,baseline=-0.5ex]
\node [inner sep=0.8pt,outer sep=0.8pt] at (1,0) (1) {$\bullet$};
\node [inner sep=0.8pt,outer sep=0.8pt] at (2,0) (2) {$\bullet$};
\draw (1,0.05)--(2,0.05);
\draw (1,-0.05)--(2,-0.05);
\end{tikzpicture},
\begin{tikzpicture}[scale=0.5,baseline=-0.5ex]
\node [inner sep=0.8pt,outer sep=0.8pt] at (1,0) (1) {$\bullet$};
\node [inner sep=0.8pt,outer sep=0.8pt] at (2,0) (2) {$\bullet$};
\draw (1,0.05)--(2,0.05);
\draw (1,-0.05)--(2,-0.05);
\draw [line width=0.5pt,line cap=round,rounded corners] (1.north west)  rectangle (1.south east);
\end{tikzpicture}, and
\begin{tikzpicture}[scale=0.5,baseline=-0.5ex]
\node [inner sep=0.8pt,outer sep=0.8pt] at (1,0) (1) {$\bullet$};
\node [inner sep=0.8pt,outer sep=0.8pt] at (2,0) (2) {$\bullet$};
\draw (1,0.05)--(2,0.05);
\draw (1,-0.05)--(2,-0.05);
\draw [line width=0.5pt,line cap=round,rounded corners] (1.north west)  rectangle (1.south east);
\draw [line width=0.5pt,line cap=round,rounded corners] (2.north west)  rectangle (2.south east);
\end{tikzpicture}, starting an induction. Suppose that $n\geq 3$ and that $J\neq \emptyset$. Let $1\leq i\leq n$ be minimal subject to $i\in J$. If $i>2$ then in the $\sA_{i-1}$ component of $\Gamma_{S\backslash\{i\}}$ we obtain an empty diagram with $\pi_{S\backslash\{i\}}$ acting nontrivially, a contradiction. Thus $i=1$ or $i=2$. Suppose first that $i=2$. Then $J$ contains no odd indices, for if $j\in J$ is odd, then in the $\sA_{j-1}$ component of $\Gamma_{S\backslash\{j\}}$ we must have a full diagram (by the $\sA_n$ analysis), contradicting $i=2$. If $2j$ is the maximal even index such that $2j\in J$ then every even index between $2$ and $2j$ is encircled, otherwise we obtain an empty $\sA_k$ diagram with $k>1$ with nontrivial diagram automorphism, a contradiction. Hence if $i=2$ we have $J=\{2,4,\ldots,2j\}$ for some $2j\leq n$. Now suppose that $i=1$. Let $j\geq 1$ be maximal subject to $j\in J$. Then in the $S\backslash\{j\}$ residue we obtain an $\sA_{j-1}$ diagram with nontrivial diagram automorphism and the first node encircled. Thus this subdiagram is full, and hence $J=\{1,2,\ldots,j\}$. 

Consider type $\sD_n$, with $\pi=(n-1,n)$. Suppose that either $n-1\in J$ or $n\in J$. Then $\{n-1,n\}\subseteq J$ because $J$ is $\pi$-invariant. If $n$ is odd then then consideration of the $\sA_{n-1}$ residue $\Gamma_{S\backslash\{n\}}$ shows that $J=\{1,2,\ldots,n\}$ and if $n$ is even then consideration of the $\sA_{n-2}$ residue $\Gamma_{S\backslash\{n-1,n\}}$ shows again that $J=\{1,2,\ldots,n\}$. Thus the diagram is full. 

Suppose now that $\{n-1,n\}\cap J=\emptyset$. Let $1\leq i\leq n-2$ be maximal subject to $i\in J$. If $i$ is even, then the induced diagram automorphism in the $\sD_{n-i+1}$ component of $\Gamma_{S\backslash\{i\}}$ is nontrivial yet the diagram is empty, a contradiction. Thus $i$ is odd, in which case the nontrivial diagram automorphism in the $\sA_{i-1}$ component of $\Gamma_{S\backslash\{i\}}$ forces $J=\{1,2,\ldots,i\}$.  

The $\sD_n$ case with $\pi=\id$ is similar to the $\sB_n$ case, using similar reasoning to the previous paragraph to show that if $J=\{1,2,\ldots,i\}$ then $i$ is necessarily even.

Consider type $\sD_4$, with $\pi=(1,3,4)$. Since $\pi\neq \id$ either $2\in J$ or $\{1,3,4\}\subseteq J$. In the first case, the residue is of type $\sA_1\times\sA_1\times\sA_1$ with $\pi_{\{2\}}$ cyclically permuting the components. Thus $\{1,3,4\}$ is also encircled. 

Consider type $\sE_6$, with $\pi=\id$. The minimal sets invariant under $w_0\circ\pi$ are $\{2\}$, $\{4\}$, $\{3,5\}$, and $\{1,6\}$. If $\{1,6\}\subseteq J$ then the diagram $\Gamma_{S\backslash \{1,6\}}$ is a $\sD_4$ with graph automorphism $(3,5)$ (in the induced labelling), and hence $2\in J$ by the above analysis of $\sD_n$. If $\{3,5\}\subseteq J$ then $\Gamma_{S\backslash\{3,5\}}$ is of type $\sA_2\times\sA_1\times\sA_1$ with the induced diagram automorphism acting nontrivially on the $\sA_2$ component and interchanging the two $\sA_1$ components, and hence $J=S$. If $4\in J$ then $\Gamma_{S\backslash\{4\}}$ has an $\sA_2\times\sA_2$ subdiagram in which the induced diagram automorphism acts by the permutation $(1,5)(3,6)$. In particular we have $\{1,3,5,6\}\cap J\neq \emptyset$. Since $J$ is $w_0$-invariant we have either $\{1,6\}\subseteq J$ or $\{3,5\}\subseteq J$, and $\pi_{\{4\}}$ invariance of $J\backslash\{4\}$ implies that $\{1,3,5,6\}\subseteq J$. In particular $\{3,5\}\subseteq J$, and so $J=S$ by the above argument.

The remaining cases are all similar. For example consider type $\sE_7$ with $\pi=\id$. Note first that if either $2\in J$ or $5\in J$ then necessarily $J=S$ and the diagram is full. For if $2\in J$ then the diagram $\Gamma_{S\backslash\{2\}}$ is of type $\sA_6$ with nontrivial diagram automorphism, and is hence full, and if $5\in J$ then the diagram $\Gamma_{S\backslash\{5\}}$ is of type $\sA_4\times\sA_2$ with nontrivial diagram automorphism on each component, and hence is full. Furthermore, note that if $6\in J$ then $\{1,6\}\subseteq  J$ (because in the residue we have a $\sD_5$ with nontrivial diagram automorphism), and if $7\in J$ then $\{1,6,7\}\subseteq J$ (because the $\sE_6$ residue has nontrivial diagram automorphism). Finally, if either $3\in J$ or $4\in J$ then $\{1,3,4,6\}\subseteq J$, for if $4\in J$ we look at the $\sA_2\times\sA_3$ component of the residue, and if $3\in J$ then the $\sA_5$ residue implies that $4\in J$.  
\end{proof}

\begin{remark} 
We note the following striking fact. The admissible diagrams listed in Table~\ref{table:1}--Table~\ref{table:5} are valid Tits diagrams (Tits indices) for Galois descent in groups of algebraic origin, with the following exceptions: The first $\sD_4$ diagram, the diagrams for $\sG_2=\mathsf{I}_2(6)$ with $\pi=\mathrm{id}$ and only one node circled, the diagram $\mathsf{I}_2(8)$ with $\pi\neq \mathrm{id}$, and obviously, due to lack of groups of algebraic origin, the cases $\mathsf{I}_2(m)$ with $m\neq 2,3,4,6,8$. The remaining diagrams are precisely the Tits diagrams for which Galois groups of order 2 exist. Moreover, we conjecture that in each of these cases there exists an involution with the given opposition diagram (but it is \emph{never} conjugate to any Galois involution). This provides a mysterious connection that warrants deeper investigation.    
\end{remark}

We can now prove Theorem~\ref{thm:main2}. 

\begin{proof}[Proof of Theorem~\ref{thm:main2}]
By Proposition~\ref{prop:conn} the opposition diagram of a capped automorphism is admissible. Now apply Theorem~\ref{thm:diagrams}.
\end{proof}

\subsection{Displacement}

We now show that the displacement of a capped automorphism can be computed from its opposition diagram, proving Theorem~\ref{thm:main3}. This fact is obviously false in general -- for example, a nondomestic duality of $\sA_2(2)$ has the same opposition diagram as the uncapped exceptional domestic duality from Example~\ref{ex:fano}. The former has displacement~$3$, while the latter has displacement~$2$. 

\begin{lemma}\label{lem:types}
If $\theta$ maps a type $J$ simplex $\sigma$ to an opposite simplex, then for all chambers $C$ containing $\sigma$ we have $\delta(C,C^{\theta})\in W_{S\backslash J}w_0$. 
\end{lemma}

\begin{proof}
By definition there exists a chamber $A$ containing $\sigma$ and a chamber $B$ containing $\sigma^{\theta}$ with $\delta(A,B)=w_0$. Since $C$ is contained in the $S\backslash J$ residue of $A$, and $C^{\theta}$ is contained in the $(S\backslash J)^{\theta}$ residue of $B$, we have $\delta(C,C^{\theta})\in W_{S\backslash J}w_0W_{(S\backslash J)^{\theta}}$. Since $\sigma$ is mapped to an opposite simplex we have $J^{\theta}=J^{\mathrm{op}}$, and thus $(S\backslash J)^{\theta}=(S\backslash J)^{\mathrm{op}}$. Therefore $w_0W_{S\backslash J}w_0=W_{(S\backslash J)^{\theta}}$ and the result follows. 
\end{proof}

The following theorem immediately implies Theorem~\ref{thm:main3}. 

\begin{thm}
Let $\theta$ be a capped automorphism with diagram $(\Gamma,J,\pi)$ and let $C$ be a chamber. Then $\ell(\delta(C,C^{\theta}))=\disp(\theta)$ if and only if $\delta(C,C^{\theta})=w_{S\backslash J}w_0$. In particular,
$$
\mathrm{disp}(\theta)=\ell(w_0)-\ell(w_{S\backslash J})=\mathrm{diam}(W)-\mathrm{diam}(W_{S\backslash J}).
$$
\end{thm}

\begin{proof}
Let $K=S\backslash J$. Necessarily $J^{\theta}=J^{\op}$, and so $K^{\theta}=K^{\op}$. Since $\theta$ is capped there exists a simplex $\sigma$ of type $J$ in $\Opp(\theta)$. For any chamber $D$ containing $\sigma$ we have $\delta(D,D^{\theta})\in W_Kw_0$ by Lemma~\ref{lem:types}, and thus
$$
\disp(\theta)\geq \ell(w_0)-\ell(w_{K}).
$$ 

Now let $C$ be any chamber such that $\ell(\delta(C,C^{\theta}))$ is maximal. By the arguments of \cite[Lemma~2.4 and Theorem~4.2]{AB:09} we have $\delta(C,C^{\theta})=w_{K'}w_0$ for some $K'\subseteq S$ with $K'^{\theta}=K'^{\op}$. Hence the type $J'=S\backslash K'$ simplex of $C$ is mapped to an opposite simplex, and hence $J'\subseteq J$. Thus $K'\supseteq K$, and so $\ell(w_{K'})\geq \ell(w_K)$, and hence
$$
\disp(\theta)=\ell(\delta(C,C^{\theta}))=\ell(w_{K'}w_0)=\ell(w_0)-\ell(w_{K'})\leq \ell(w_0)-\ell(w_K)\leq \disp(\theta),
$$
and the result follows.
\end{proof}

For example, the possible displacements of capped collineations of $\sE_7$ buildings are $0$, $33$, $50$, $51$, $60$, and $63$, corresponding to the $6$ opposition diagrams stated in the introduction.

\section{Automorphisms of large classical buildings}\label{sec:3}

In this section we prove Theorem~\ref{thm:main} for classical buildings (types $\sA$, $\sB=\sC$, and $\sD$; note that since we use Coxeter graphs rather than Dynkin diagrams, we do not make a distinction between $\sB_n$ and $\sC_n$ buildings). Thus in this section we prove:

\begin{thm}
Every automorphism of a large classical building of rank at least $3$ is capped.
\end{thm}

The case of polar spaces turns out to be particularly challenging, and requires a series of lemmas presented in Section~\ref{sec:lemmas}. We prove some of these results in greater generality than is strictly required for this paper -- for example, where possible we prove some results for small buildings. This additional generality never requires much additional work, and will be useful in future work on uncapped automorphisms.

\subsection{Buildings of type $\sA_n$}\label{sect:An}

Buildings of type $\sA_n$ play an important role in our proof techniques owing to their prevalence as residues of spherical buildings of arbitrary type. It is therefore a pleasant and useful fact that domesticity in buildings of type $\sA_n$ is a very well behaved phenomenon. In this section we show that all automorphisms of large $\sA_n$ buildings are capped. We also prove that every exceptional domestic duality of $\sA_n(2)$ is strongly exceptional domestic. 

Every thick building of type $\sA_n$ with $n>2$ is a projective space $\mathsf{PG}(n,\KK)$ over a division ring~$\KK$, where the type $i$ vertices of the building are the $(i-1)$-spaces of the projective space. Thus points have type~$1$, lines have type $2$, and so on. 

\begin{defn}
Let $(\cdot,\cdot)$ be a nondegenerate symplectic form on $\FF^{2n}$, with $\FF$ a field. Write $U^{\circ}=\{v\mid (u,v)=0\text{ for all $u\in U$}\}$. A duality $\theta$ of $\mathsf{PG}(2n-1,\FF)$ of the form $U^{\theta}=U^{\circ}$ is called a \textit{symplectic polarity}. 
\end{defn}

We recall the following lemma from \cite[Lemma~3.2]{TTM:11}. 

\begin{lemma}\label{lem:An}  If the projective space $\Delta=\mathsf{PG}(n,\KK)$ admits a duality $\theta$ for which all points are absolute (equivalently no type~$1$ vertex is mapped to an opposite), then $n$ is odd, $\KK$ is a field, and $\theta$ is a symplectic polarity.  
\end{lemma}

It is simple to show that symplectic polarities are capped. In fact we have the following stronger result.

\begin{lemma}\label{lem:sp}
If $\theta$ is a symplectic polarity of an $\sA_{2n-1}$ building $\Delta$ then $\theta$ is $\{i\}$-domestic for each odd~$i$, and each vertex mapped to an opposite vertex is contained in a type $\{2,4,\ldots,2n-2\}$ simplex mapped to an opposite simplex.
\end{lemma}

\begin{proof}
Let $\Delta=\mathsf{PG}(2n-1,\FF)$. Spaces $U$ and $U'$ are opposite if and only if $U+U'=V$ and $U\cap U'=\{0\}$, where $V=\FF^{2n}$. Let $U$ be an $(i-1)$-space (that is, a type $i$ vertex). If $i$ is odd then the symplectic form $(\cdot,\cdot)$ defined by $\theta$ is necessarily degenerate when restricted to $U$, since $U$ has odd vector space dimension. Thus $U\cap U^{\circ}\neq\{0\}$, and so $U+U^{\theta}\neq V$, and hence $\theta$ is $\{i\}$-domestic.

Suppose that $i$ is even, and that the $(i-1)$-space $U$ is mapped to an opposite. Thus $U\cap U^{\circ}=\{0\}$, and so $(\cdot,\cdot)$ restricted to $U$ is nondegenerate. Hence there is a symplectic basis $e_1,\ldots,e_{i/2},f_1,\ldots,f_{i/2}$ of~$U$, and extending this to a symplectic basis $e_1,\ldots,e_{n},f_1,\ldots,f_{n}$ of $V$ we see that the flag $(U_2,U_4,\ldots,U_{2n})$ with $U_{2j}=\langle e_1,\ldots,e_j,f_1,\ldots,f_j\rangle$ is mapped to an opposite flag by~$\theta$. 
\end{proof}

Part of the following theorem is contained in \cite[Theorem~3.1]{TTM:11}, however there are some oversights there for the buildings~$\sA_n(2)$, and so we provide a proof. We extend the definition of small buildings to include $\sA_2$ buildings by declaring $\sA_2(2)$ ``small'', and all other thick $\sA_2$ buildings ``large''.

\begin{thm}\label{thm:Alarge}
Let $\theta$ be a domestic duality of a large building of type $\sA_n$ with $n\geq 2$. Then $n$ is odd and $\theta$ is a symplectic polarity. 
\end{thm}

\begin{proof}
The base case $n=2$ is Theorem~\ref{thm:rank2}. Suppose that $n>2$ and that the result holds for all $k<n$. If $\theta$ is $\{n\}$-domestic then by the dual of Lemma~\ref{lem:An} $n$ is odd and $\theta$ is a symplectic polarity. So suppose that $\theta$ maps a type $n$ vertex $v$ to an opposite vertex. By Corollary~\ref{cor:proj} the induced automorphism $\theta_v$ of the $\sA_{n-1}$ building $\Res(v)$ is a domestic duality, and thus $n$ is even and $\theta_v$ is a symplectic polarity (by the induction hypothesis). Hence, by Lemma~\ref{lem:sp} and Proposition~\ref{prop:proj}, there is a type $\{2,n\}$ simplex of $\Delta$ mapped to an opposite simplex by $\theta$. If $v'$ is the type $2$ vertex of this simplex then $\theta_{v'}$ acts as a duality on the $\sA_{n-2}$ component of $\Res(v')$ mapping a type $n$ vertex to an opposite (in the induced labelling). Thus by the induction hypothesis $\theta_{v'}$ is not domestic, and hence there is a type $\{2,3,\ldots,n\}$ simplex of $\Delta$ mapped to an opposite by $\theta$. By considering the residue of the type $n$ vertex of this simplex we conclude that $\theta$ is not domestic, a contradiction. 
\end{proof}

\begin{cor}\label{cor:An1}
Every duality of a large building of type $\sA_n$ is capped. 
\end{cor}

\begin{proof}
Every non-domestic duality is automatically capped, and by Theorem~\ref{thm:Alarge} every domestic duality is a symplectic polarity, and is hence capped by Lemma~\ref{lem:sp}.
\end{proof}

\begin{lemma}\label{lem:AnAn}
Let $n\geq 1$. No automorphism of a thick building of type $\sA_n\times\sA_n$ interchanging the two components and acting as an involution on the type set is domestic. 
\end{lemma}

\begin{proof}
If $n=1$ let $v$ be any vertex of the first component, and let $v'$ be a vertex of the second component with $v'\neq v^{\theta},v^{\theta^{-1}}$ (by thickness). Then the chamber $\{v,v'\}$ is mapped to an opposite chamber, starting an induction. Suppose that $n\geq 2$. By relabelling the types in the second component, we may assume that the type map of $\theta$ interchanges $i$ and $n+i$ for each $1\leq i\leq n$. Since $\theta$ is not trivial it must map some simplex to an opposite simplex. It follows that there is a simplex $\sigma$ of type $\{i,n-i+1,n+1,2n-i+1\}$ mapped to an opposite, for some $1\leq i\leq n/2$. The residue of $\sigma$ is a building of type $\mathsf{X}\times \mathsf{X}$ where $\mathsf{X}=\sA_{i-1}\times \sA_{n-2i}\times \sA_{i-1}$ (with some components of $\mathsf{X}$ empty if $i=1$ or $i= n/2$), and $\theta_{\sigma}$ interchanges the two type $\mathsf{X}$ components and acts as an involution on the types. Thus, by induction, there is a chamber of $\Res(\sigma)$ mapped to an opposite chamber by $\theta_{\sigma}$, and hence by Proposition~\ref{prop:proj} $\theta$ is not domestic on~$\Delta$. 
\end{proof}

\begin{lemma}\label{lem:Ancol}
Let $\Delta$ be a thick building of type $\sA_n$ and let $\theta$ be a domestic collineation. If there exists a simplex of type $\{i,n-i+1\}$ mapped to an opposite simplex by $\theta$ then there exists a simplex of type $\{j,n-j+1\mid 1\leq j\leq i\}$ mapped to an opposite simplex by $\theta$.
\end{lemma}

\begin{proof}
If there exists a type $\{i,n-i+1\}$ simplex $\sigma$ mapped to an opposite, then the induced automorphism of the type $\sA_{i-1}\times \sA_{n-2i}\times \sA_{i-1}$ building $\Res(\sigma)$ interchanges the $\sA_{i-1}$ components, and is a collineation on the $\sA_{n-2i}$ component. Thus by Lemma~\ref{lem:AnAn} there is a chamber of the $\sA_{i-1}\times \sA_{i-1}$ subbuilding mapped to an opposite, and the result follows from Proposition~\ref{prop:proj}.
\end{proof}

\begin{cor}\label{cor:An2} Every collineation of a thick type $\sA_n$ building is capped.
\end{cor}

\begin{proof} Suppose that $\theta$ is a nontrivial collineation. Let $1\leq i\leq n/2$ be maximal subject to there being a type $\{i,n-i+1\}$ simplex mapped to an opposite simplex by $\theta$. By Lemma~\ref{lem:Ancol} there is a type $\{j,n-j+1\mid 1\leq j\leq i\}$ simplex mapped to an opposite, and hence $\theta$ is capped. 
\end{proof}

Hence the proof of the main theorem for buildings of type $\sA_n$ is complete. We conclude this section with two results on dualities of small $\sA_n$ buildings.

\begin{thm}\label{thm:Asmall}
Let $\theta$ be a domestic duality of the small building $\Delta=\sA_n(2)$ with $n\geq 2$. Then either $\theta$ is an exceptional domestic duality or $n$ is odd and $\theta$ is a symplectic polarity. 
\end{thm}

\begin{proof}
The base case $n=2$ is Theorem~\ref{thm:rank2}. Suppose that $n>2$ and that the results hold for all $k<n$. If $\theta$ is $\{n\}$-domestic then by the dual of Lemma~\ref{lem:An} $n$ is odd and $\theta$ is a symplectic polarity. So suppose that $\theta$ maps a type $n$ vertex $v$ to an opposite vertex. By Corollary~\ref{cor:proj} the induced automorphism $\theta_v$ of the $\sA_{n-1}$ building $\Res(v)$ is a domestic duality. Then either $\theta_v$ is exceptional domestic, in which case it is clear that $\theta$ is also exceptional domestic, or $n$ is even and $\theta_v$ is a symplectic polarity. Hence there is a type $\{2,n\}$ simplex of $\Delta$ mapped to an opposite simplex by $\theta$. If $v'$ is the type $2$ vertex of this simplex then $\theta_{v'}$ acts as a duality on the $\sA_{n-2}$ component of $\Res(v')$ mapping a type $n$ vertex to an opposite (in the induced labelling). Thus $\theta_{v'}$ is either not domestic, or $\theta_{v'}$ is exceptional domestic, and in either case it follows that $\theta$ is exceptional domestic.  
\end{proof}

\begin{thm}\label{thm:Astex}
Every exceptional domestic duality of an $\sA_n(2)$ building is strongly exceptional domestic. 
\end{thm}

\begin{proof}
We begin with the following claim: \textit{If a duality $\theta$ of $\sA_n(2)$ with $n>2$ is $\{1,n\}$-domestic then it is either $\{1\}$-domestic or $\{n\}$-domestic.} We show that, if a duality $\theta$ maps a point to an opposite, then it maps an incident point-hyperplane pair to an opposite. Thus let $p$ be a point mapped to an opposite, and set $H=p^\theta$. Also, set ${p'}=H^\theta$ and ${H'}^\theta=p$. The duality $\theta_H$ that maps a point $r$ of $H$ onto the $(n-2)$-space $r^\theta\cap H$ is point-domestic (if $r\in H$ were mapped to an opposite, then $\{r,H\}$ is mapped to an opposite).

Let $q$ be a point distinct from $p$ and not in $H\cup H'$. Then, since $p\notin H\cup H'$, the unique third point $r$ on the line $pq$ belongs to $H\cap H'$. As $\theta_H$ is point-domestic, $r^\theta$ contains $r$, and since $r\in H'$, $r^\theta$ also contains $p$. Since $p^\theta$ does not contain $p$, neither $q^\theta$ contains $p$. But both $p^\theta$ and $r^\theta$ contain $r$, hence $r\in q^\theta$. It follows that $q\notin q^\theta$, so every point not in $H\cup H'$ is mapped to an opposite. Dually, every hyperplane not containing either $p$ or $p'$ is mapped to an opposite. Since $n>2$, there are at least three such hyperplanes, and thus there exists a hyperplane $G$ mapped to an opposite distinct from $H$ and distinct from $H'$. Suppose that $G$ is contained in $H\cup H'$. Clearly $G$ is not contained in $H\cap H'$, so assume without loss that $G$ contains a point $x\in H\setminus H'$. If it also contains a point $x'\in H'\setminus H$, then the third point on the line $xx'$ is in $G$ but not in $H\cup H'$, a contradiction. So $G\subseteq H$ and $G=H$, a contradiction again.  Hence $G$ contains a point $r$ outside $H\cup H'$, and then $\{r,G\}$ is mapped to an opposite, hence the claim.

The result now follows by induction. If $n=2$ then the two notions coincide, so suppose that $n>2$. If $\theta$ is an exceptional domestic duality then by the claim there exists a type $\{1,n\}$ simplex mapped onto an opposite simplex. In the residue of the type $n$ vertex of this simplex we obtain a domestic duality of $\sA_{n-1}$ mapping a type $1$ vertex to an opposite. Hence this duality is exceptional domestic, and hence strongly exceptional domestic by the induction hypothesis. It follows, using Proposition~\ref{prop:proj}, that there exist panels of each cotype $1,2,\ldots,n-1$ of $\Delta$ mapped to opposites. A symmetric argument looking at the type $1$ vertex gives cotype $2,3,\ldots,n$ panels mapped to opposites, and hence $\theta$ is strongly exceptional domestic. 
\end{proof}

\subsection{Buildings of type $\sB_n$ and $\sD_n$}

In this section we show that all automorphisms of large buildings of type $\sB_n$ and $\sD_n$ are capped. The main additional ingredients to the residue arguments are listed in the following proposition. 

\begin{prop}\label{prop:summary} Let $n\geq 3$. 
\begin{compactenum}[$(1)$]
\item Let $\Delta$ be a large building of type $\sB_n$ or $\sD_{n+1}$ and let $i<n$. If $\theta$ is a $\{1,i\}$-domestic collineation then either $\theta$ is $\{1\}$-domestic or $\{i\}$-domestic.
\item Let $\Delta$ be building of type $\sB_n$ with thick projective space residues. If $\theta$ is a $\{1,n\}$-domestic collineation then $\theta$ is either $\{1\}$-domestic or $\{n\}$-domestic. 
\item Let $\Delta$ be a thick building of type $\sD_{n}$ with $n$ even and let $\theta$ be a collineation.
\begin{compactenum}[$(a)$]
\item If $\theta$ is $\{1,n\}$-domestic then $\theta$ is either $\{1\}$-domestic or $\{n\}$-domestic. 
\item If $\theta$ is $\{n-1,n\}$-domestic, then $\theta$ is either $\{n-1\}$-domestic or $\{n\}$-domestic.
\end{compactenum}
\item Let $\Delta$ be a large building of type $\sD_{n}$ with $n$ odd and let $\theta$ be a collineation. If $\theta$ is $\{1,n-1,n\}$-domestic then $\theta$ is either $\{1\}$-domestic or $\{n-1,n\}$-domestic.
\end{compactenum}
\end{prop}

Proposition~\ref{prop:summary} is a natural extension of the analysis of domesticity in polar spaces initiated in~\cite{TTM:12}. The proof of the proposition is somewhat involved, and so we temporarily postpone the proof and continue with the main theorem.

\begin{thm}\label{thm:Bn1} Every collineation of a large building of type $\sB_n$ with $n\geq 3$ is capped. 
\end{thm}

\begin{proof}
Let $\theta$ be a nontrivial collineation, and let $J=\Type(\theta)$. Let $j\in J$ be maximal, and let $v$ be a vertex of type $j$ mapped to an opposite vertex by $\theta$. If $j=1$ then $\theta$ is obviously capped. If $j=2$ and $J=\{2\}$ then $\theta$ is obviously capped, and if $j=2$ and $J=\{1,2\}$ then by Proposition~\ref{prop:summary}(1) $\theta$ is capped. So assume that $j\geq 3$. Then $\theta_v$ acts on the type $\sA_{j-1}$ component of $\Res(v)$ as a duality. Suppose that $j$ is odd. Then by Theorem~\ref{thm:Alarge} $\theta_v$ is not domestic on the $\sA_{j-1}$ component of $\Res(v)$ and so $\theta_v$ maps a type $\{1,2,\ldots,j-1\}$ simplex $\sigma$ of $\Res(v)$ to an opposite. Thus by Proposition~\ref{prop:proj} the simplex $\sigma\cup\{v\}$ is mapped to an opposite simplex by~$\theta$, and so $J=\{1,2,\ldots,j\}$ and $\theta$ is capped. If $j$ is even then either $\theta_v$ is non-domestic, in which case $J=\{1,2,\ldots,j\}$ and we are done as before, or $\theta_v$ is symplectic polarity, in which case it follows from Lemma~\ref{lem:sp} and Proposition~\ref{prop:proj} that there is a type $\{2,4,\ldots,j-2,j\}$ simplex mapped to an opposite by~$\theta$. If $J=\{2,4,\ldots,j-2,j\}$ then $\theta$ is capped. Otherwise  there exists $j'\in J$ odd, and then by the argument above there is a type $1$ vertex mapped to an opposite. Hence by Proposition~\ref{prop:summary}~(1) or (2) (in the cases $j<n$ or $j=n$) we have a type $\{1,j\}$ simplex mapped to an opposite simplex by~$\theta$. If $v'$ is the type $j$ vertex of this simplex then $\theta_{v'}$ acts as a duality on the $\sA_{j-1}$ component of $\Res(v')$ mapping a point to an opposite point, and thus by Theorem~\ref{thm:Alarge} and Proposition~\ref{prop:proj} we have $J=\{1,2,\ldots,j\}$ and $\theta$ is capped. 
\end{proof}

\begin{thm}\label{thm:Dn1} Every collineation of a large building of type $\sD_n$ is capped. 
\end{thm}

\begin{proof}
Let $\theta$ be a nontrivial collineation, and let $J=\Type(\theta)$. Let $j\in J$ be maximal. If $j\leq n-2$ then the argument from Theorem~\ref{thm:Bn1} implies that $\theta$ is capped. So suppose that $j=n-1$ or $j=n$. By symmetry we may suppose that $j=n$. Thus if $n$ is even there is a type $n$ vertex $v$ mapped to an opposite vertex, and if $n$ is odd then there is a type $\{n-1,n\}$ simplex $\sigma$ mapped to an opposite simplex. We consider each case below.

Suppose that $n$ is even. Then $\theta_v$ is a duality of the type $\sA_{n-1}$ building $\Res(v)$, and thus by Theorem~\ref{thm:Alarge} and Proposition~\ref{prop:proj} there is a type $\{2,4,\ldots,n-2,n\}$ simplex mapped to an opposite by~$\theta$. If $J=\{2,4,\ldots,n-2,n\}$ then we are done. If there exists $j'\in J$ with $j'$ odd and $j'\leq n-2$ then an argument as in the proof of Theorem~\ref{thm:Bn1} implies that $1\in J$, and thus by Proposition~\ref{prop:summary}(3)(a) there is a type $\{1,n\}$ simplex mapped to an opposite. If $v'$ is the type $n$ vertex of this simplex then $\theta_{v'}$ is a duality of the $\sA_{n-1}$ building $\Res(v')$ mapping a type~$1$ vertex to an opposite. Thus $\theta_{v'}$ is not domestic, and so $\theta$ is not domestic (by Proposition~\ref{prop:proj}), and hence $\theta$ is capped. If $n-1\in J$ then by Proposition~\ref{prop:summary}(3)(b) there is a type $\{n-1,n\}$ simplex $\sigma'$ mapped to an opposite. Then $\theta_{\sigma'}$ is a duality of the $\sA_{n-2}$ building $\Res(\sigma')$, and hence by Theorem~\ref{thm:Alarge} and Proposition~\ref{prop:proj} $\theta$ is not domestic, and hence is capped. 

Suppose that $n$ is odd. Then $\Res(\sigma)$ is an $\sA_{n-2}$ building and it follows that there is a type $\{2,4,\ldots,n-3,n-1,n\}$ simplex mapped to an opposite. If $J=\{2,4,\ldots,n-3,n-1,n\}$ then we are done. Otherwise there is $j'\in J$ with $j'\leq n-2$ and $j'$ odd then as above we have $1\in J$. Hence by Proposition~\ref{prop:summary}(4) there is a type $\{1,n-1,n\}$ simplex mapped to an opposite, and it follows that $\theta$ is not domestic, and hence is capped.
\end{proof}

\begin{thm}\label{thm:DD} Every duality of a large building of type $\sD_n$ is capped. 
\end{thm}

\begin{proof}
Let $\theta$ be a duality, and let $J=\Type(\theta)$. Let $j\in J$ be maximal. Suppose that $j\leq n-2$, and let $v$ be a type $j$ vertex mapped to an opposite. If $j$ is even then $\theta_v$ is a duality of the $\sD_{n-j}$ component of $\Res(v)$ (note that this is true irrespective of the parity of $n$, see Example~\ref{ex:D}), and hence maps some simplex of this residue to an opposite, contradicting maximality of $j$ (using Proposition~\ref{prop:proj}). Thus $j$ is odd. Since $\theta_v$ is a duality on the $\sA_{j-1}$ component of $\Res(v)$ it follows from Theorem~\ref{thm:Alarge} and Proposition~\ref{prop:proj} that there is a type $\{1,2,\ldots,j\}$ simplex mapped to an opposite by~$\theta$, and hence $J=\{1,2,\ldots,j\}$ and $\theta$ is capped. 

If $j=n-1$ or $j=n$ then by symmetry we may assume that $j=n$. If $n$ is odd then there is a type $n$ vertex $v'$ mapped to an opposite vertex by $\theta$, and considering the induced duality of the $\sA_{n-1}$ building $\Res(v')$ and applying Theorem~\ref{thm:Alarge} and Proposition~\ref{prop:proj} we see that $\theta$ is not domestic, and hence is capped. If $n$ is even then there is a type $\{n-1,n\}$ simplex mapped to an opposite by $\theta$, and again we conclude that $\theta$ is not domestic, and hence capped. 
\end{proof}

A similar argument to Theorem~\ref{thm:DD} proves the following useful fact, which we record for later use.

\begin{prop}\label{prop:DDual}
No duality of a thick building of type $\sD_n$ is $\{1\}$-domestic.
\end{prop}

\begin{proof}
Let $J=\Type(\theta)$. Let $j\in J$ be maximal. If $j\leq n-2$ then the argument of Theorem~\ref{thm:DD} shows that $j$ is odd, and if $v$ is a type $j$ vertex mapped to an opposite then Theorem~\ref{thm:Asmall} implies that $\theta_v$ is either non-domestic or exceptional domestic on the $\sA_{j-1}$ component of $\Res(v)$. In either case $1\in J$ (using Proposition~\ref{prop:proj}). Similar arguments to Theorem~\ref{thm:DD} show that if $j=n-1$ or $j=n$ then $1\in J$ and hence the result. 
\end{proof}

The final case to consider is trialities of $\sD_4$. 

\begin{thm}\label{thm:DTri} Every triality of a thick building of type $\sD_4$ is capped. 
\end{thm}

\begin{proof}
Let $\theta$ be a triality and let $J=\Type(\theta)$. Suppose that $2\in J$, and let $v$ be a type $2$ vertex mapped to an opposite vertex. Then $\theta_v$ is an automorphism of a thick $\sA_1\times\sA_1\times\sA_1$ building cyclically permuting the components, and so clearly $\theta_v$ is not domestic. Hence $\theta$ is not domestic, and hence is capped. If $2\notin J$ then $J=\{1,3,4\}$ and $\theta$ is obviously capped. 
\end{proof}

\begin{cor}
Every automorphism of a large building of type $\sB_n$ or $\sD_n$ with $n\geq 3$ is capped. 
\end{cor}

\begin{proof}
This follows from Theorems~\ref{thm:Bn1}, \ref{thm:Dn1}, \ref{thm:DD}, and~\ref{thm:DTri}. 
\end{proof}

\subsection{Proof of Proposition~\ref{prop:summary}}\label{sec:lemmas}

Of course it now remains to prove Proposition~\ref{prop:summary}. For this purpose it is much better to work in the geometric language of polar spaces; a brief summary can be found in \cite[Section~2]{TTM:12}, see also \cite[Chapter~7]{DeB:16}. 

We will always assume a polar space to have thick lines, i.e., every line contains at least three points. We call a polar space of rank $n$ \emph{thick} if every $(n-2)$-dimensional singular subspace is contained in at least three maximal singular subspaces (of dimension $n-1$), and non-thick otherwise. In the non-thick case every $(n-2)$-dimensional singular subspace is contained in exactly two maximal singular subspaces, and there are two types of maximal singular subspaces (members of the same type intersect each other in subspaces of even codimension). The so-called \emph{oriflamme complex} of such a polar space is then a thick building of type $\mathsf{D}_n$. Conversely, every thick building of type $\mathsf{D}_n$ gives rise to a non-thick polar space of rank $n$. Most results we will prove in this section hold for thick and non-thick polar spaces, and we will draw conclusions for buildings of types $\mathsf{B}_n$ and $\mathsf{D}_n$ from these results. 

In the language of polar spaces it is convenient to use ``(projective) dimension as a singular subspace'' rather then ``type of a vertex'', and therefore there is a shift in indexing. That is, ``$i$-domesticity'' in the polar space indexing will mean $\{i+1\}$-domestic in the building indexing. Similarly, $(i,i+1)$-domestic in the polar space means $\{i-1,i\}$-domestic in the building indexing. There is an exception for polar spaces of type $\sD_n$, where there are two types of maximal singular subspaces. These both have projective dimension $n-1$, and we will call them $(n-1)'$-spaces and $(n-1)$-spaces (corresponding to vertices of types $n-1$ and $n$ in the building). The expressions \textit{point-domestic}, \textit{line-domestic}, and so on, have the obvious meanings. 

%

If $U$ is an $i$-space of a polar space we speak of the ``lower residue'', consisting of all spaces contained in $U$ and the ``upper residue'', consisting of all spaces which contain~$U$. We write $\theta_{\subseteq U}$ and $\theta_{\supseteq U}$ for the induced automorphisms of the lower and upper residue, respectively. Moreover, it is convenient at times to use terminology like ``$U$ is domestic'' and ``$U$ is non-domestic'' as short hand for ``$\theta$ does not map $U$ to an opposite'' and ``$U$ is mapped to an opposite by $\theta$''.

We begin by recalling some facts from~\cite{TTM:12}. However the proof of Theorem~3.2 of \cite{TTM:12} makes use of the wrong fact that if a plane $\pi$ of a polar space $\Delta$ is mapped to an opposite plane by $\theta$ then the duality $\theta_{\subseteq \pi}$ is not domestic (this may fail if the plane is Fano). There are two ways in which we can revise this result: The first is to assume that the planes of $\Delta$ are not Fano planes, and the second is to assume that $\theta$ is plane-domestic. Hence the two revisions below.

\begin{fact}[Lemma 3.1 of \cite{TTM:12}]\label{fact1}
A collineation of a polar space $\Delta$ which is both point-domestic and line-domestic is the identity.
\end{fact}

\begin{fact}[Revision of Theorem~3.2 of \cite{TTM:12}]\label{fact2}
A collineation of a polar space $\Delta$ of rank at least~$3$ without Fano plane residues which is $(\mbox{point},\mbox{line})$-domestic is either point-domestic or line-domestic.
\end{fact}

\begin{fact}[Revision of Theorem~3.2 of \cite{TTM:12}]\label{fact3}
A collineation of a polar space $\Delta$ of rank at least $3$ which is both $(\mbox{point},\mbox{line})$-domestic and plane-domestic is either point-domestic or line-domestic.
\end{fact}

\begin{fact}[Theorem 6.1 of \cite{TTM:12}]\label{fact4}
A collineation of a polar space $\Delta$ of rank~$n\geq 3$ which is both $i$-domestic and $(i+1)$-domestic, $0\leq i<n-1$, fixes at least one point in every $i$-space. 
\end{fact}

\subsubsection{Proof of Proposition~\ref{prop:summary}(1)}

Recall that we allow non-thick polar spaces, however we always assume polar spaces to have thick lines.

\begin{lemma}\label{ieven}
Let $\theta$ be a collineation of a polar space $\Delta$ of rank $n\geq 3$. If $\theta$ is $(0,i)$-domestic with $2\leq i<n$ and $i$ even, then $\theta$ is $i$-domestic. 
\end{lemma}

\begin{proof}
Suppose $\theta$ is not $i$-domestic and let $U$ be an $i$-space mapped onto an opposite. Then $\theta_{\subseteq U}$  is a point-domestic duality of $U$. Hence, as $i\geq 2$, Lemma~\ref{lem:An} implies that $\theta_{\subseteq U}$ is a symplectic polarity, contradicting $i$ being even. 
\end{proof}

\begin{lemma}
Let $\theta$ be a collineation of a polar space $\Delta$ of rank $n\geq 3$. If $\theta$ is $(0,i)$-domestic with $3\leq i<n-1$ and $i$ odd then $\theta$ is $(i+\ell)$-domestic for all odd positive $\ell<n-i$. 
\end{lemma}

\begin{proof}
Suppose $\theta$ is not $(i+\ell)$-domestic for some odd positive $\ell<n-i$, and let $U$ be an $(i+\ell)$-space mapped onto an opposite. Then $\theta_{\subseteq U}$ is a domestic duality of the projective space $U$. Since $\dim U=i+\ell$ is even, this duality is exceptional domestic (by Theorem~\ref{thm:Asmall}). However, by Theorem~\ref{thm:Astex}, every exceptional domestic duality of a projective space is strongly exceptional domestic, contradicting the $(0,i)$-domesticity of $\theta$ and the fact that $i+\ell\geq 4$. 
\end{proof}

\begin{lemma}\label{lemmaconjugate} 
Let $\theta$ be a collineation of a polar space of rank $n\geq 3$. Let $0\leq i<n-1$. Suppose $U$ and $U'$ are two adjacent non-domestic $i$-spaces (adjacent means that they are contained in a common $(i+1)$-space $W$), and suppose that $W$ contains a fixed point $p$ for $\theta$. Then the dualities $\theta_U$ and $\theta_{U'}$ are conjugate. In particular, they are either both domestic or both non-domestic. 
\end{lemma}
\begin{proof}
By the assumptions we know that $p\notin U\cap U'$. Consider the isomorphism $\varphi:U\rightarrow U'$ mapping $u\in U$ onto $\<p,u\>\cap U'$. Then we show that $u^{\theta_U\varphi}=u^{\varphi\theta_{U'}}$, which will prove the lemma. 

Now, $u^{\theta_U}=\proj_U u^\theta=\proj_Wu^\theta\cap U$, whereas similarly $u^{\varphi\theta_{U'}}=\proj_Wu^{\varphi\theta}\cap U'$. Since $u,u^\varphi$ and $p$ are collinear, and since $p^\theta=p$, also $u^\theta, u^{\varphi\theta}$ and $p$ are collinear. Hence, since $p\in W$, $$\proj_Wu^\theta=(u^\theta)^{\perp}\cap W=(u^{\varphi\theta})^\perp\cap W=\proj_Wu^{\varphi\theta},$$ and this set contains $p$. Hence $(\proj_Wu^\theta\cap U)^\varphi=\proj_Wu^{\varphi\theta}\cap U'$. This also holds if $u=u'\in U\cap U'$. Hence the lemma is proved.
\end{proof}

\begin{lemma}\label{flagtypeline}
Let $\Delta$ be a projective space of dimension $n\geq 2$ over some division ring $\mathbb{K}$, and let $\theta$ be a duality of $\Delta$. Let $0\leq \ell\leq n-1$, with $|\mathbb{K}|>2$ if $(\ell,n)=(0,2)$, and suppose that there exists some non-domestic flag $(V_i)_{0\leq i\leq\ell}$ for $\theta$, with $V_{i-1}\subseteq V_{i}$, for all $i=1,2,\ldots,\ell$, and $\dim V_i=i$, for all $i\in\{0,1,\ldots,\ell\}$. Let $L$ be any line of $\Delta$. Then  there exists some non-domestic flag $(U_i)_{0\leq i\leq\ell}$ for $\theta$, with $U_{i-1}\subseteq U_{i}$, for all $i=1,2,\ldots,\ell$, $\dim U_i=i$, for all $i\in\{0,1,\ldots,\ell\}$, and $L$ not incident with any $U_i$, $0\leq i\leq\ell$.
\end{lemma}

\begin{proof}
Suppose $i\leq\ell$ is maximal with respect to the property that $V_i$ does not contain $L$. Then $i$ is well defined since $V_0$ does not contain $L$. Also, if $i=\ell\geq 1$, then we only need to show that, in case $V_0\in L$, we can select a second non-domestic point $U_0\in V_1$  distinct from $V_0$. But this follows from the fact that $\theta_{V_1}$ is not the identity (as $V_0$ is not fixed) and hence has at least two non-fixed points $V_0$ and $U_0$. So we may assume that $0\leq i\leq \ell-1$ or $i=\ell=0$. First suppose $0\leq i\leq \ell-1$. We claim that there exists a non-domestic $(i+1)$-space $U_{i+1}$, with $V_i\leq U_{i+1}\leq U_{i+2}$ (if $i=\ell-1$, then $U_{i+2}=\Delta$), with $L$ not contained in $U_{i+1}$.

Indeed,  suppose first $i\leq \ell-2$. The spaces $V_{i+2}$ and $V_i^\theta$ intersect in a line $M$. Then $M\cap V_{i+1}$ is a point $p$ and $M\cap V_{i+1}^\theta$ is some point $p'$. Since $V_{i+1}$ is opposite $V_{i+1}^\theta$, we have $p\neq p'$. Set $U_{i+1}=\<V_i,p'\>$. Then $U_{i+1}^\theta$ intersects $M$ in a point distinct from $p'$ and is hence opposite $U_{i+1}$. Thus, since $\<V_i,L\>=V_{i+1}$, $L$ is not contained in $U_{i+1}$ and the claim follows.

Next, suppose $i=\ell-1$. If $\ell=n-1$, then putting $V_{i+2}=\Delta$, we can copy the argument of the previous paragraph. So we may assume that $\ell\leq n-2$. The duality in $V_{\ell-1}^\theta$ induced by projection onto $V_{\ell-1}$ followed by $\theta$ has a non-domestic point $V_\ell\cap V_{\ell-1}^\theta$. If there is any other non-domestic point, say $q$, then $U_{\ell}=\<V_{\ell-1},q\>$ does the job. 

Hence assume that a duality $\varphi$ of a projective space $\Gamma$ has at least one non-domestic point~$p$. We assert that there is a second one. If $p\neq p^{\varphi^2}$, then $p^{\varphi^2}$ is a second non-domestic point. So suppose $p=p^{\varphi^2}$, that is, $p^{\varphi^{-1}}=p^\varphi$. If $\dim\Gamma=1$, then the assertion is trivial since non-domesticity is the same as non-fixed, and one cannot fix everything except precisely one point. Hence $\dim\Gamma\geq 2$ and we can consider a line $K$ through $p$. Since $p\notin p^\varphi$, we deduce $K\not\subseteq K^\varphi$. It follows that for at most one point $r\in K$ the property $r\in r^\varphi$ holds (and then $r=K\cap K^\varphi$). Hence all points on $K$ except for this possible $r$ are non-domestic. This completes the proof of the claim. 

If $\ell>0$, then an obvious induction argument now completes the proof of the lemma.

Finally, assume that   $i=\ell=0$. Then we have to prove that not all non-domestic points of $\theta$ are contained in the line $L$, given that $L$ contains at least one such, say $p$, and assuming $|\mathbb{K}|>2$ if $n=2$. Let $U=p^\theta$ and consider the duality $\varphi$ defined on $U$ by first projecting on $p$ and then applying $\theta$. Since above we have shown that a duality is either point-domestic or has at least two non-domestic points, we may assume that every point of $U$ is domestic. But if $U^\theta=p$ then, as above, every line through $p$ contains at least two non-domestic points, and we are done. Hence we may assume that $U^\theta=r\in L\setminus(U\cup\{p\})$. Let $L'$ be a line through $p$ distinct from $L$. Since $L'^\theta\subseteq U$, there is exactly one point $s$ on $L'$ such that $L'\subseteq s^{\perp}$. If $|\mathbb{K}|>2$, then it follows that there is a point $t\in L'\setminus(U\cup\{p,s\})$, and it is easy to see that $t$ is non-domestic. Hence we may assume that $|\mathbb{K}|=2$, and consequently $n\geq 3$. Suppose $L=\{p,q,r\}\neq L'=\{p,q',r'\}$, with $q,q'\in U$ and $q'\in q^\theta$ (this is possible since $n>2$). Since $p$ and $r$ are the only non-domestic points for $\theta$, we deduce $p^{\theta^{-1}}=r^\theta$. 
Since $q'\in q^\theta\cap p^\theta$, we have $q'\in r^\theta=p^{\theta^{-1}}$, and so $p\in q'^\theta$. 
Since $q'\in q'^\theta$, we then have $r'\in q'^\theta$. Since also $r'$ is domestic, we have $r'\in r'^\theta$. Then $r'\in q'^\theta\cap r'^\theta=L'^\theta\subseteq p^\theta$, which is a contradiction as $r'\notin U=p^\theta$. This completes the proof of the lemma.
\end{proof}

A flag $(V_i)_{0\leq i\leq\ell}$, with $V_{i-1}\subseteq V_{i}$, for all $i=1,2,\ldots,\ell$, and $\dim V_i=i$, for all $i\in\{0,1,\ldots,\ell\}$, will be referred to as a \emph{flag of type $[0,\ell]$}. 

An almost identical proof as the previous one shows the following lemma (the reason why the exception for $|\mathbb{K}|=2$ and $(\ell,n)=(0,2)$ is not turning up is because we know that every duality of the Fano plane contains at least two non-domestic points).

\begin{lemma}\label{flagtypepoint}
Let $\Delta$ be a projective space of dimension $n\geq 2$ over some division ring $\mathbb{K}$, and let $\theta$ be a duality of $\Delta$. Let $0\leq \ell\leq n-1$ and suppose that there exists some non-domestic flag of type $[0,\ell]$. Let $x$ be any point of $\Delta$. Then  there exists some non-domestic flag of type $[0,\ell]$ no element of which contains $x$.\qed
\end{lemma}

\begin{lemma} \label{basicnewlemma}
Let $\theta$ be a collineation of a polar space $\Delta$ of rank $n\geq 3$. Let $0\leq i<n-2$. Suppose that there exists an $i$-space $U_0$ mapped onto an opposite. Let either $\theta$ be $(i,i+1)$- and $(i,i+2)$-domestic, or suppose there is an integer $\ell$,  $0\leq \ell\leq i-1$ such that $U_0$ contains a non-domestic flag of type $[0,\ell]$ and, whenever some $i$-space $U$ is mapped onto an opposite, and contains a non-domestic flag of type $[0,\ell]$ for $\theta_{\subseteq U}$, then $\theta_{\supseteq U}$ is the identity. Then every $(i+1)$-space has a fixed point. In particular, $\theta$ is $j$-domestic for every $j\geq i+1$.
\end{lemma}

\begin{proof}
By the hypothesis, $\theta_{\supseteq U_0}$ is both $(i+1)$-domestic and $(i+2)$-domestic, and hence is the identity by Fact~\ref{fact1}. Thus every singular subspace $W$ containing $U_0$ is mapped onto the unique singular subspace through $U_0^\theta$ containing all points of  $W$ collinear to all points of $U_0^\theta$.

First we show that 
every $(i+1)$-space containing $U_0$ contains a fixed point. Suppose for a contradiction that some $(i+1)$-space $Z$ containing $U_0$ contains no fixed point. 
Let $\{p\}=Z\cap Z^\theta$. 
Then, clearly, $p\perp p^\theta$. Let $L$ be the preimage of the line $pp^\theta$. Then $L$ is contained in $Z$ and contains $p$. Let $W$ be an $(i+2)$-space containing $Z$ and let $M$ be the preimage of the line $N=W\cap W^\theta$. 

We first claim that under the assumption that if $|\mathbb{K}|=2$ (where $\mathbb{K}$ is the division ring of the projective space residues) then $i>2$, the lines $M$ and $N$ intersect nontrivially. Indeed, suppose not. Let $K$ be the line obtained by intersecting $U_0$ with the subspace of $W$  generated by $N$ and~$M$.

By Lemma~\ref{flagtypeline}, we can select an $(i-1)$-space $V_0$ in $U_0$ not containing $K$, such that, if $\theta$ is not both $(i,i+1)$-domestic and $(i,i+2)$-domestic, then $V_0$ contains a non-domestic flag of type $[0,\ell]$. Now we can easily choose an $i$-space $U'$ in $W$ containing $V_0$ and disjoint from $M\cup N$. Without loss of generality, we may assume that $V_0$ is also disjoint from $L$, as we can otherwise interchange the roles of $Z$ and another $(i+1)$-space in $W$ containing $U_0$. Since $U'$ does not intersect $M$, the image $U'^\theta$ does not intersect $W$. Also, the subspace of $W$ spanned by $U'$ and $N$ is $W$. Hence a point $a$ of $U'^\theta$ collinear with all of $U'$ is collinear with all of $W$, hence with all of $U_0$. Hence all of $U_0$ is collinear with all of the plane spanned by $N$ and $a$, and this plane intersects $U_0^\theta$ in some point $b$. Then $b\in U_0^\theta$ is collinear with all of $U_0$, contradicting the fact that $U_0$ and $U_0^\theta$ are opposite. So, we have verified that $U'^\theta$ is opposite $U'$. By the choice of $V_0$, our assumptions imply that $\theta_{\supseteq U'}$ is the identity. Consequently, the subspace $Z'$ generated by $U'$ and $p$ is mapped onto the subspace $Z'^\theta$ generated by $V_0^\theta$ and $p$. Since $Z'$ contains $p$, the subspace $Z'^\theta$ contains $p^\theta$. And since $Z'^\theta$ also contains $p$, it contains $L^\theta$. So $Z'$ contains $L$, contradicting our choice of $Z'$.  Hence $p$ is fixed. So $p\in M\cap N$ and our claim is proved. 

We next claim that, if $N$ and $M$ intersect non-trivially, then $N=M$ and so $p=p^\theta$. Indeed, suppose $q=N\cap M$. We may assume $p\neq q$. Then clearly $q$ is a fixed point for $\theta$, as the subspace spanned by $U_0$ and $q$ is mapped onto the one spanned by $U_0^\theta$ and $q$, and the intersection of the former with $M$ (being $\{q\}$) is mapped onto the intersection with $M^\theta=N$ (also being $\{q\}$). Now by Lemma~\ref{flagtypepoint}, we can select an $(i-1)$-space $V_0$ in $U_0$ disjoint from $L$, such that, if $\theta$ is not both $(i,i+1)$-domestic and $(i,i+2)$-domestic, then $V_0$ contains a non-domestic flag of type $[0,\ell]$. Then we choose an $i$-space $U'$ through $V_0$ not containing $p$. Completely similar to the previous paragraph one shows that $U'$ is opposite $U'^\theta$, that $\theta_{\supseteq U'}$ is the identity, and that this leads to $p$ being fixed. Hence the claim. 

This already shows that, if $(|\mathbb{K}|,\ell,i)\neq (2,0,2)$, then every $(i+1)$-space through $U_0$ contains a fixed point. 

Now suppose $(|\mathbb{K}|,\ell,i)= (2,0,2)$. Observe that, by connectivity of the polar space $U_0^\perp\cap (U_0^{\theta})^\perp$, and by our previous claim, it suffices to show the existence of at least one $(i+1)$-space (that is, a $3$-space) containing $U_0$ and containing a fixed point for $\theta$. Note that the rank of $\Delta$ is at least $i+3=5$, and $U_0$ is a Fano plane. There are three cases. 
\begin{itemize} 
\item Either $\Delta$ is a symplectic polar space $\sB_n(2)$ with $n\geq 5$. Then $\Delta$ is embedded, in a standard way, into a projective space of dimension $d$ at least $9$. The perp of a plane (the subspaces consisting of all points collinear to that pane) has dimension $d-3$. Hence $U_0^\perp\cap (U_0^\theta)^\perp$ has dimension at least $d-6$. Then $U_0^\perp\cap(U_0^\theta)^\perp\cap(U_0^{\theta^2})^\perp$ has dimension at least $d-9\geq 0$. Hence there exists a point $x\in U_0^\perp\cap(U_0^\theta)^\perp\cap(U_0^{\theta^2})^\perp$. Let $Z$ be the singular subspace of dimension $i+1$ spanned by $U_0$ and $x$. Then, as before, $Z$ is mapped onto the $(i+1)$-space spanned by $U_0^\theta$ and $x$. Likewise, $Z^\theta$ is mapped onto the $(i+1)$-space spanned by $U_0^{\theta^2}$ and~$x$. It follows that $x^\theta=(Z\cap Z^\theta)^\theta=Z^\theta\cap Z^{\theta^2}=x$. Hence $x$ is fixed, and the assertion follows. 
\item Or $\Delta$ is the $\mathsf{B}_n(2,4)$ polar space, with $n\geq 5$. Then $\Delta$ can be represented as a
quadric in $(2n+1)$-dimensional projective space, with $2n+1\geq 11$. Similarly as above, the intersection of the perp of three planes is now a subspace $\pi$ of dimension at least $(2n+1)-9\geq 2$. Since no plane of the ambient projective space is disjoint with a quadric, we again obtain a fixed point in some $(i+1)$-space $Z$ containing $U_0$. By our previous arguments, this implies the assertion. 
\item Or $\Delta$ is the non-thick polar space associated to $\sD_n(2)$ with $n\geq 5$. Then $\Delta$ can be represented as a quadric in $(2n-1)$-dimensional  projective space. Thus if $n\geq 6$ an argument similar to the previous bullet point applies. For the case $n=5$ we have verified the lemma by direct computation and we omit the details.
\end{itemize}

Next we show that every $(i+1)$-space contains a fixed point. To that aim, we note that, if some $i$-subspace $U$ is mapped to an opposite, and if $\theta_{\subseteq U}$ is conjugate to $\theta_{\subseteq U_0}$, then these dualities have the same properties; in particular, if $U_0$ contains a flag of type $[0,\ell]$ mapped to an opposite for $\theta_{\subseteq U_0}$, then so does $U$ for $\theta_{\subseteq U}$, and the above part of the proof implies that every $(i+1)$-space containing $U$ has a fixed point. 

We assume for a contradiction that some $(i+1)$-space $Z$ has no fixed point. 
Let $k$ be maximal with respect to the property that there exists an $i$-space $U$ which is mapped onto an opposite, such that $\theta_{\subseteq U}$ is conjugate to $\theta_{\subseteq U_0}$, and which intersects $Z$ in a $k$-space. If $k=i$, then the foregoing paragraph leads to a contradiction. Now suppose $k<i$ (note that we allow $k$ to equal $-1$). Since $\dim Z>\dim U$, we find a point $x\in Z\setminus U$ collinear to all points of $U$. Let $Z'$ be the $(i+1)$-space spanned by $U$ and $x$. By the previous paragraph, $Z'$ has a unique fixed point $p$ (namely, $Z'\cap {Z'}^\theta$), and it is not contained in $Z$. In $Z'$, we find at least one $i$-subspace $U'$ containing $Z\cap Z'$ and avoiding $p$. Now $U'$ is mapped onto an opposite, $\theta_{\subseteq U'}$ is conjugate to $\theta_{\subseteq U}$ by Lemma~\ref{lemmaconjugate}, and $\dim(U'\cap Z)=\dim(U\cap Z)+1$, contradicting the maximality of $k$. Hence $Z$ contains a fixed point after all.    
\end{proof}

We now mention the special cases of Lemma~\ref{basicnewlemma} for $(i,i+1)$- and $(i,i+2)$-domestic collineations. After that, we prepare a more serious application, namely, for $\ell=0$. An application for $\ell=i-1$ will be presented in~\cite{PVM:17b} when we study automorphisms of small buildings. 

\begin{cor}\label{ii+1ii+2}
Let $\theta$ be a collineation of a polar space $\Delta$ of rank $n\geq 3$. Suppose that $\theta$ is both $(i,i+1)$-domestic and $(i,i+2)$-domestic for some $0\leq i\leq n-3$. Then either $\theta$ is $i$-domestic, or every $(i+1)$-space has at least one fixed point. In particular, if $\theta$ is not $i$-domestic, then $\theta$ is $j$-domestic for all $j$ with $i+1\leq j\leq n-1$.\qed
\end{cor}


\begin{lemma}\label{i-1dom}
Let $\theta$ be a collineation of a polar space $\Delta$ of rank $n\geq 3$. If $\theta$ is $(0,i)$-domestic, $3\leq i<n-1$ and $i$ odd, but not $i$-domestic, then $\theta$ is $(i-1)$-domestic. 
\end{lemma}

\begin{proof}
Suppose for a contradiction that some $(i-1)$-space $U$ is mapped onto an opposite. Since $i-1$ is even, the duality $\theta_{\subseteq U}$ in $U$ is not point-domestic. Hence $\{0,i\}$-domesticity of $\theta$ implies that no $i$-space containing $U$  is mapped onto a opposite. This means that $\theta$ is $(i-1,i)$-domestic. Since projective spaces of even dimension either do not admit domestic dualities, or, by Theorem~\ref{thm:Astex}, only strongly exceptional domestic ones, $\theta$ is $(i+1)$-domestic and hence $(i-1,i+1)$-domestic, Lemma~\ref{ii+1ii+2} implies that $\theta$ is $i$-domestic, a contradiction. 
\end{proof}

\begin{lemma}\label{induction}
Let $\theta$ be a collineation of a polar space $\Delta$ of rank $n\geq 3$. If $\theta$ is $(0,i)$-domestic for some $3\leq i<n-1$ with $i$ odd, but not $i$-domestic, then $\theta$ is $(0,i-2)$-domestic. 
\end{lemma}

\begin{proof}

Suppose for a contradiction that $\theta$ maps some $\{p,U\}$ to an opposite, with $U$ an $(i-2)$-space and $p$ a point in $U$. Then $(0,i)$-domesticity of $\theta$ implies that no $i$-space containing $U$ is mapped onto an opposite. Also, by Lemma~\ref{i-1dom}, $\theta$ is $(i-1)$-domestic. Hence $U$ satisfies the conditions of Lemma~\ref{basicnewlemma} for $\ell=0$ and $i$ replaced by $i-2$. Hence every $i$-space contains fixed points, which implies that $\theta$ is $i$-domestic, a contradiction to our hypotheses.
\end{proof}

We are now ready to prove Proposition~\ref{prop:summary}(1), in the following formulation:

\begin{lemma}\label{1-i}
Let $\theta$ be a collineation of a polar space $\Delta$ of rank $n\geq 3$, assumed to be thick if $n=3$. If $\theta$ is $(0,i)$-domestic for some $1\leq i<n-1$, and if the planes of $\Delta$ are not Fano planes, then $\theta$ is either point-domestic or $i$-domestic. 
\end{lemma}

\begin{proof}
We argue by induction on $i$. If $i=1$, then this is Fact~\ref{fact2}. Now suppose $i>1$. Suppose that $\theta$ is not $i$-domestic. Then we have to show that $\theta$ is point-domestic. Lemma~\ref{ieven} implies that $i$ is odd. So Lemma~\ref{induction} implies that $\theta$ is $(0,i-2)$-domestic. Induction implies that $\theta$ is either point-domestic, in which case we are done, or $(i-2)$-domestic. Lemma~\ref{i-1dom} implies that $\theta$ is $(i-1)$-domestic. Fact~\ref{fact4} now implies that every $(i-2)$-space has at least one fixed point, and hence $\theta$ is $j$-domestic for $i-2\leq j\leq n-1$. In particular $\theta$ is $i$-domestic, a contradiction. Hence $\theta$ is point-domestic after all. 
\end{proof}

\begin{cor}
Let $\theta$ be a collineation of a polar space $\Delta$ of rank $n\geq 3$, assumed to be thick if $n=3$. If $\theta$ is $(0,i)$-domestic for some $1\leq i<n-1$, and also plane-domestic, then $\theta$ is either point-domestic or $i$-domestic. 
\end{cor}

\begin{proof}
In view of Lemma~\ref{1-i}, we only need to show this for polar spaces having Fano planes. In the proof of Lemma~\ref{1-i}, there is only one place where we need the planes not to be Fano planes, and that is in the first step, the base of the induction. However, Fact~\ref{fact3} states that the base is valid if $\theta$ is plane-domestic.
\end{proof}

\subsubsection{Proof of Proposition~\ref{prop:summary}(4)}

We now turn to Proposition~\ref{prop:summary}(4). The following lemma allows us to translate $((n-1)',(n-1))$-domesticity in the oriflamme complex (a type $\sD_n$ building) to $(n-2)$-domesticity in the associated non-thick polar space (a type $\sB_n$ building). 

\begin{lemma}\label{lem:containedB}
Let $\Delta$ be a non-thick polar space of rank $n$ with $n$ odd. A collineation $\theta$ of $\Delta$ is $(n-1,(n-1)')$-domestic (on the oriflamme complex) if and only if $\theta$ is $(n-2)$-domestic (on~$\Delta$).  
\end{lemma}

\begin{proof}
Suppose that $(U,V)$ is an incident $(n-1,(n-1)')$ pair of subspaces of the oriflamme complex, and that $(U,V)^{\theta}$ is opposite $(U,V)$. Let $W=U\cap V$ be the associated $(n-2)$-space of $\Delta$. Then $U$ is opposite $V^{\theta}$ and $V$ is opposite $U^{\theta}$. Thus $U$ and $V^{\theta}$ are disjoint, and so $W$ and $W^{\theta}$ are disjoint. If there is a point $p$ of $W^{\theta}$ collinear with all points of $W$ then $\langle p,W\rangle$ is either $U$ or $V$, contradicting the fact that $U$ is opposite $V^{\theta}$ and $V$ is opposite $U^{\theta}$. Thus $W$ and $W^{\theta}$ are opposite in $\Delta$.

Conversely, suppose that $W$ is an $(n-2)$-space of $\Delta$ mapped onto an opposite space by $\theta$. 
In particular, $W$ and $W^{\theta}$ are disjoint. Let $(U,V)$ be the $n-1$ and $(n-1)'$-spaces associated to~$W$. Thus $U\cap V=W$. Since $n$ is odd and $\theta$ is a collineation the projective dimensions of $U\cap V^{\theta}$ and $V\cap U^{\theta}$ are odd. Since $W$ and $W^{\theta}$ are disjoint this forces $U$ and $V^{\theta}$ to be disjoint, and $V$ and $U^{\theta}$ to be disjoint, and hence $(U,V)$ is opposite $(U,V)^{\theta}$ in the oriflamme complex.
\end{proof}

We can now prove Proposition~\ref{prop:summary}(4), which, in the present language, reads as follows. 

\begin{lemma}
Every $(0,n-1,(n-1)')$-domestic collineation $\theta$ of an oriflamme complex of odd rank~$n$ with no Fano plane residues is either point-domestic or $(n-1,(n-1)')$-domestic.
\end{lemma}

\begin{proof}
By Lemma~\ref{lem:containedB} we see that $\theta$ is $(0,n-2)$-domestic in the associated non-thick polar space. Thus by Lemma~\ref{1-i} is follows that $\theta$ is either point-domestic or $(n-2)$-domestic, hence the result.
\end{proof}

\subsubsection{Proof of Proposition~\ref{prop:summary}(2)}

Let $\Delta$ be a polar space of rank $n\geq 3$ (recall that we only assume thick lines). By an  \textit{oppomorphism} of $\Delta$ we shall mean
\begin{compactenum}[$\bullet$]
\item any collineation in the thick case, and
\item if $\Delta$ is non-thick, a collineation of $\Delta$ that induces an oppomorphism in the associated thick building of type $\sD_n$. 
\end{compactenum}
The main result of this subsection is Lemma~\ref{lemma1-n} below, from which Proposition~\ref{prop:summary}(2) is immediate.

\begin{lemma}\label{lemma1-n}
Let $\theta$ be an oppomorphism of a polar space $\Delta$ of rank $n\geq 3$. If $\theta$ is $(0,n-1)$-domestic then $\theta$ is either $(n-1)$-domestic, or point-domestic. 
\end{lemma}

Note that Lemma~\ref{ieven} takes care of the case when $n$ is odd. So we may assume that $n\geq 4$ is even. Before proving Lemma~\ref{lemma1-n} we require three additional lemmas. 

\begin{lemma}\label{lemmaaffine}
Let $\Delta$ be a projective space of dimension $n\geq 2$ over some division ring $\mathbb{K}$, with $|\mathbb{K}|>2$, and let $\theta$ be a non-domestic duality of $\Delta$. Then the set of non-domestic points is not contained in a hyperplane of $\Delta$. 
\end{lemma}

\begin{proof}
We give an inductive proof on $n$. First let $n=2$. Suppose that the set of non-domestic points is contained in the line $L$. Then every line not incident with $L^{\theta}$ is domestic. Let $M\neq L$ be such a line. Then it is easy to see that all points of $M\setminus\{M^{\theta}, M^{\theta^{-1}}\}$ are non-domestic, contradicting the fact that only $M\cap L$ is possibly non-domestic and $|M\setminus\{M^{\theta}, M^{\theta^{-1}}\}|\geq 2$. 

Now let $n\geq 3$. Suppose that the set of non-domestic points is contained in the hyperplane $H$. Select a non-domestic chamber $C$ and let $p\in H\cap C$. Then $\theta_p$ is a non-domestic duality in the residue of $p$, and so not all non-domestic lines through $p$ can be contained in $H$, by the induction hypothesis. Hence there exists some non-domestic line $L\ni p$, with $L\notin H$. All points of $L\setminus\{p\}$ are domestic, hence satisfy $q^\theta=\<L^\theta,q\>$, which implies by exhaustion that $p^\theta=\<L^\theta,p\>$, contradicting the fact that $p$ is non-domestic. 
\end{proof}

\begin{lemma}\label{lemma11-n}
Let $\theta$ be an oppomorphism of a polar space $\Delta$ of rank $n\geq 3$. Suppose that $n$ is even, and that $\theta$ is $(0,n-1)$-domestic. Suppose $\theta$ maps some flag $\{p,U\}$ to an opposite, where $p$ is a point and $U$ an $(n-3)$-space. Then $\theta$ fixes pointwise either a geometric hyperplane of the rank 2 polar space $U^\perp\cap (U^\theta)^\perp$, or $\theta$ fixes every point of this rank~$2$ polar space.
\end{lemma}

\begin{proof}
By $(0,n-1)$-domesticity, $\theta_{\supseteq U}$ is line-domestic. Hence, if $\Delta$ is non-thick, then $\theta_{\supseteq U}$ is clearly the identity. If $\Delta$ is thick, then $\theta_{\supseteq U}$ is either the identity, or pointwise fixes a proper geometric hyperplane with the following property: If $\theta_{\supseteq U}$ fixes a line, then all points of that line belong to that geometric hyperplane (see \cite{TTM:12b}). We note that there are three kinds of geometric hyperplanes of a generalized quadrangle: the set of points collinear to a certain point (in which case every fixed point is contained in a pointwise fixed line), the set of points of a large full subquadrangle (where the adjective ``full'' means that the subquadrangle is a subspace and the adjective  ``large'' just means that it is a geometric hyperplane; also in this case every fixed point is contained in a pointwise fixed line), and the set of points of an ovoid (i.e., a geometric hyperplane without lines; here no point is contained in a fixed line). In the next paragraph we will consider the case where no fixed point is contained in a fixed line. Hence this only applies to the case of an ovoid. Since in a generalized quadrangle of order $(2,2)$ no nontrivial collineation fixes an ovoid pointwise, and since a generalized quadrangle of order $(2,4)$ does not admit ovoids (see \cite[Theorem~1.8.3]{PT:09}), we may suppose that the underlying field has at least $3$ elements. 

Now let $H$  be the set of points of $U^\perp\cap (U^\theta)^\perp$ that corresponds to the set of fixed points of $\theta_{\supseteq U}$. Let $x\in H$. We first assume that $\theta_{\supseteq U}$ does not fix any $(n-1)$-space containing $x$ (hence $H$ is an ovoid). Let $Z$ be such an $(n-1)$-space, and let $W$ be the singular $(n-2)$-subspace containing $U$ and $x$. Suppose for a contradiction that $x':=x^\theta\neq x$. Then also $x'':=x^{\theta^{-1}}\ne x$.  Since $\theta_{\subseteq U}$ is not point-domestic (as $p$ is non-domestic), and the underlying field has at least $3$ elements, by (the dual form of) Lemma~\ref{lemmaaffine} we can find a non-domestic $(n-4)$-space $V\subseteq U$ disjoint from the line $\<x,x''\>$, i.e.~$V$ is not through the point $u:=\<x,x''\>\cap U$ (note that $\dim U>2$ or $U$ is a line, in which case $U$ contains at least two non-domestic points one of which must be different from $u$). Since $n-4$ is even, $V$ contains a non-domestic point $p_0$. Let $W'\neq W$ be an $(n-2)$-space in $Z$ containing $U$. Since $U$ does not contain $x''$, nor will $W'$ contain it. Since $W'\neq W$, the former does not contain $x$ either (but $u\in W'$). Hence $W'^\theta$ does not contain $x$, and if some point of it is collinear with all points of $W'$, then it is collinear with all points of $Z$, hence belongs to $Z$, hence coincides with $x$, a contradiction. So $W'$ is non-domestic. Hence there is a second non-domestic $(n-3)$-space $U_0$, $V\subseteq U_0\subseteq W'$ apart from $U$. Since $u\in U$ and $u\notin V$, we deduce $u\notin U_0$. Note $p_0\in U_0$. Hence $\theta_{\supseteq U_0}$ is line domestic, but does not fix the line corresponding to $Z$. Since $Z\cap Z^\theta=\{x\}$, the only $(n-2)$-space  through $U_0$ that is possibly fixed by  $\theta_{\supseteq U_0}$ is the one containing $x$, call it $W_0$. Hence $x\in W_0^\theta$, so $x''\in W_0$. But then $xx''\subseteq W_0$ and so $xx''$ intersects $U_0\subseteq W'$, necessarily in $u$, a contradiction as we already noted above that $u\notin U_0$. This contradiction shows that $x=x'=x''$. 

Now we forget the previous notation and suppose that  $\theta_{\supseteq U}$  fixes some $(n-1)$-space $Z$ containing $x$. Since $Z$ is fixed under $\theta_{\supseteq U_0}$, for every non-domestic $(n-3)$-space $U_0$ in $Z$, every $(n-2)$-space containing $U_0$ and contained in $Z$ will be fixed by $\theta_{\supseteq U_0}$ as soon as $U_0$ contains a non-domestic point. Following the third paragraph of the proof of Lemma~\ref{basicnewlemma} we obtain that $\theta$ fixes~$x$. 

We conclude from the two previous paragraphs that $H$ is pointwise fixed by $\theta$. 
\end{proof}

\begin{lemma}\label{lemma21-n}
Let $\theta$ be an oppomorphism of a polar space $\Delta$ of rank $n\geq 3$. Suppose that $n$ is even, and that $\theta$ is $(0,n-1)$-domestic. If $\theta$ is not $(n-1)$-domestic then $\theta$ is $(0,n-3)$-domestic and $(n-2)$-domestic.
\end{lemma}

\begin{proof}
Suppose for a contradiction that $\theta$ is not $(0,n-3)$-domestic. Let $Z$ be a non-domestic $(n-1)$-space chosen in such a way that amongst all non-domestic $(n-1)$-spaces, $k:=\dim (Z\cap U)$, with $U$ a non-domestic $(n-3)$-space containing a non-domestic point, is maximal. Note that $k$ is well defined by the assumption that $\theta$ is not $(0,n-3)$-domestic. Suppose $k<n-3$. The following is a slight variation of the last paragraph of the proof of Lemma~\ref{basicnewlemma}. Let $Z'$ be the unique $(n-1)$-space containing $U$ and intersecting $Z$ in a $(k+2)$-space. Lemma~\ref{lemma11-n} implies that $Z'$ contains some fixed point $z$. Since $Z$ is non-domestic we have $z\notin Z$. Let $W$ be the $(n-2)$-space generated by $U$ and $z$, and let $U'$ be any $(n-3)$-space in $W$ not through $z$ and intersecting $Z$ in the $(k+1)$-space $W\cap Z$. Then one easily verifies that $U'$ is non-domestic, and by Lemma~\ref{lemmaconjugate}, $U'$ contains a non-domestic point. This contradicts the maximality of $k$. Hence $k=n-3$ and $Z$ contains a fixed point by Lemma~\ref{lemma11-n}. Hence $Z$ is domestic.

These contradictions show that $\theta$ is $(0,n-3)$-domestic. Suppose now an $(n-2)$-space $V$ is non-domestic. Then, since $n-2$ is even, $V$ contains a $\{0,n-3\}$-flag mapped to an opposite, a contradiction. Hence $\theta$ is $(n-2)$-domestic. 
\end{proof}

We can now prove Lemma~\ref{lemma1-n}, and hence Proposition~\ref{prop:summary}(2).

\begin{proof}[Proof of Lemma~\ref{lemma1-n}]
Recall that we may assume $n\geq 4$ is even. Let $\theta$ be $(0,n-1)$-domestic, and suppose that $\theta$ is neither point-domestic nor $(n-1)$-domestic.  By Lemma~\ref{lemma21-n} the automorphism $\theta_{\supseteq p}$ is $i$-domestic for $i=n-3,n-2,n-1$. We take the natural number $j\leq n-4$ maximal with respect to the property that there exists a non-domestic $j$-space containing a non-domestic point (we allow $j=0$). Then Lemma~\ref{basicnewlemma} for $\ell=0$ and $i=j$ shows that 
every $(j+1)$-space contains a fixed point. Since $j+1\leq n-3$, this implies that every $(n-1)$-space contains a fixed point, contradicting the assumption that $\theta$ is not $(n-1)$-domestic. 
%
%
%
\end{proof}

\subsubsection{Proof of Proposition~\ref{prop:summary}(3)(a)}

In the polar space language Proposition~\ref{prop:summary}(3)(a) reads as follows.

\begin{lemma}\label{lemma1-nD}
Let $\Delta$ be the oriflamme complex of a non-thick polar space of rank $n\geq 4$ with $n$ even. Every $(0,n-1)$-domestic collineation of $\Delta$ is either point-domestic or $(n-1)$-domestic.
\end{lemma}

\begin{proof}
Suppose some point-$(n-1)'$-space pair $(p,U)$ is mapped to an opposite. Then, since $\theta_{U}$ is not $0$-domestic, there exists some non-domestic $(n-1)$-space $U'$ incident with $p$ and $U$. This contradicts the $(0,n-1)$-domesticity of $\theta$. Hence $\theta$ is also $(0,(n-1)')$-domestic. Consequently $\theta$ is $(0,n-1)$-domestic as an oppomorphism of the associated non-thick polar space. The assertion now directly follows from Lemma~\ref{lemma1-n}.
\end{proof}

\subsubsection{Proof of Proposition~\ref{prop:summary}(3)(b)}

We now turn to Proposition~\ref{prop:summary}(3)(b), which reads as follows.

\begin{lemma}\label{lemman'n}
Let $\Delta$ be the oriflamme complex of a non-thick polar space of rank~$2n\geq 4$. If a collineation $\theta$ of $\Delta$ is $(2n-1,(2n-1)')$-domestic, then it is either $(2n-1)$-domestic or $(2n-1)'$-domestic.  
\end{lemma}

Note that by triality, Lemma~\ref{lemma1-nD} for $n=4$ proves Lemma~\ref{lemman'n} for $n=2$. Hence we can give an inductive proof of Lemma~\ref{lemman'n}, where the initial step $n=2$ is already done. So we may suppose that $n>2$ and that the lemma holds for ranks strictly smaller than $2n$. Let $\theta$ be a $(2n-1,(2n-1)')$-domestic collineation.

\begin{lemma}\label{lemma1Dn}
If $U$ and $U'$ are non-domestic $(2n-1)$-spaces and $(2n-1)'$-spaces, respectively, then $U\cap U'$ is a totally singular subspace for both symplectic polarities $\theta_U$ and $\theta_{U'}$.
\end{lemma}

\begin{proof}
If $U\cap U'$ is a point, then there is nothing to prove. So suppose $U\cap U'$ has at least dimension $2$. Let for a contradiction $L$ be a line in $U\cap U'$ which is not totally singular for $\theta_U$. Then $L$ is non-domestic and the induction hypothesis shows that $\theta_U$ cannot map both $U$ and $U'$ to opposites, a contradiction.  
\end{proof}

In particular, Lemma~\ref{lemma1Dn} implies that the dimension of the intersection of two non-domestic maximal singular subspaces of distinct type is at most $n-1$. In fact we can say much more:




\begin{lemma}\label{lemma3Dn}
If $U$ and $U'$ are non-domestic $(2n-1)$-spaces and $(2n-1)'$-spaces, respectively, then $U\cap U'$ is a point.
\end{lemma}

\begin{proof}
Suppose for a contradiction that $U\cap U'$ is a subspace $V$ of dimension $2i$, with $i\geq 1$. Then there is a $(2i+1)$-space $W$ of $U$ containing $V$ and not being contained in $V^{\perp'}$, where $\perp'$ denotes collinearity in the symplectic polar space induced in $U$ by $\theta_U$. Set $T=W^{\perp'}\cap V$. Then $\dim T=2i-1$, and since $i\geq 1$, we see that $T\neq\emptyset$.   Let $Y$ be the unique $(2n-2)$-space in $U'$ all of whose points are collinear to each point of $W$ (and note that, reciprocally,  $W$ is the set of points of $U$ collinear to each point of $Y$). The domestic $(2n-3)$-spaces of $Y$ are precisely those that contain the point $Y^{\theta_{U'}}\in Y$. Hence we can select a non-domestic $(2n-3)$-space $Z\subseteq Y$ not containing $T$. Then $Z\cap V$ is $(2i-1)$-dimensional. Since $U'\supseteq Z$ is non-domestic, there is some other non-domestic $(2n-1)'$-space $U''\supseteq Z$. Since $V\cap Z\subseteq U\cap U''$ is $(2i-1)$-dimensional, we have $\dim (U\cap U'')\geq 2i$. Since $Z$ contains an $(2n-2i-3)$-dimensional subspace disjoint from $V$, we have $\dim (U\cap U'')\leq (2n-1)-(2n-2i-3)-1=2i+1$. We conclude $\dim(U\cap U'')=2i$. A point of $U\cap U''$ is collinear to all points of $Z\cup V$, hence to all points of $Y$, hence it belongs to $W$. Obviously, if $T\subseteq U''$, then $U'=U''$. So $U\cap U''$ is a $2i$-space in $W$ not containing $T$. If $U\cap U''$ were totally singular for $\theta_U$, then, since all points of the singular subspace $T$ are collinear for the symplectic polarity $\theta_U$ with all points of the singular subspace $U\cap U''$, the subspace $W$, which is generated by $U\cap U''$ and $T$ (because $i>0$), would be singular for $\theta_U$, a contradiction.  Hence $U\cap U''$ is not totally singular for $\theta_U$, contradicting Lemma~\ref{lemma1Dn}.
\end{proof}

\begin{lemma}\label{lemma4Dn}
Let $U$ be a non-domestic $(2n-1)$-space and suppose that $L$ is a non-domestic line in~$U$. Let $U_0$ be the $(2n-1)$-space containing $L$ and intersecting $U^\theta$ in a $(2n-3)$-space, say $W$. Then $U_0$ is non-domestic.
\end{lemma}

\begin{proof}
Suppose for a contradiction that $U_0\cap U_0^\theta$ contains a point $z$. Then $z$ is collinear to all points of $L^\theta\cup W$. Since $L$ is non-domestic, $L^\theta$ and $W$ are disjoint and so they generate $U^\theta$. Hence $z\in U^\theta$, and so $z^{\theta^{-1}}\in U$. But as $z\in U_0^\theta$. we see that $z^{\theta^{-1}}\in   U_0$. Clearly, $U_0\cap U=L$. Hence $z\in L^\theta$. But $L^\theta\cap (U_0\cap U^\theta) = L^\theta\cap W=\emptyset$, a contradiction.
\end{proof}

We are now ready to finish the proof of Lemma~\ref{lemman'n}.

\begin{proof}[Proof of Lemma~\ref{lemman'n}] Suppose for a contradiction that $\theta$ maps a $(2n-1)$-space $U$ to an opposite, and also a $(2n-1)'$-space $U'$. By Lemma~\ref{lemma3Dn}, $x=U\cap U'$ and $y=U^\theta\cap U'$ are points. Note that $x^\theta\neq y$. Hence there is a line $L$ in $(y^\perp\cap U)\setminus ((x^\theta)^\perp\setminus\{x\})$, with $x\in L$. Since $L$ is not in $(x^\theta)^\perp$, it is non-domestic, and so is $W:=L^\perp\cap U^\theta$.  Lemma~\ref{lemma4Dn} implies that the $(2n-1)$-space $U_0$ containing $L\cup W$ is non-domestic. But it intersects $U'$ is at least a line $xy$, contradicting Lemma~\ref{lemma3Dn}. 
\end{proof}

The proof of Proposition~\ref{prop:summary} is now complete.

\section{Automorphisms of large exceptional buildings}\label{sec:4}

In this section we conclude the proof of Theorem~\ref{thm:main} by proving cappedness of automorphisms of large buildings of exceptional types $\sE_6$, $\sE_7$, $\sE_8$, and $\sF_4$. We point out that in \cite{HVM:12}, where domestic dualities of $\sE_6$ buildings are investigated, the exceptional behaviour of the small $\sE_6(2)$ building is overlooked. In particular, the results \cite[Main Result~2.2 and Corollary~2.3]{HVM:12} hold for all large $\sE_6$ buildings, but fail for the $\sE_6(2)$ building. 

We first prove some elementary lemmas.

\begin{lemma}\label{lem:F41}
No duality of a thick $\sF_4$ building is domestic.
\end{lemma}

\begin{proof}
Suppose that $\theta$ is a domestic duality. Since $\theta$ maps some simplex to an opposite there is either a type $\{1,4\}$ simplex $\sigma_1$ mapped to an opposite simplex, or a type $\{2,3\}$ simplex $\sigma_2$ mapped to an opposite. In the former case, $\theta_{\sigma_1}$ is a domestic duality of a $\sB_2$ building, a contradiction (see Theorem~\ref{thm:rank2}). In the latter case, $\theta_{\sigma_2}$ is a domestic automorphism of an $\sA_1\times\sA_1$ building interchanging the components, a contradiction.
\end{proof}

\begin{lemma}\label{lem:exceptional} Let $\theta$ be a domestic automorphism of a large building $\Delta$ of type $\sX$. 
\begin{compactenum}[$(1)$]
\item If $\sX=\sF_4$ and $\theta$ is a collineation then $\theta$ is $\{i\}$-domestic for all $i\in \{2,3\}$. 
\item If $\sX=\sE_6$ and $\theta$ is a duality then $\theta$ is $\{i\}$-domestic for all $i\in\{2,3,4,5\}$. 
\item If $\sX=\sE_6$ and $\theta$ is a collineation then $\theta$ is $J$-domestic for all $J\in\{\{4\},\{3,5\}\}$. 
\item If $\sX=\sE_7$ then $\theta$ is $\{i\}$-domestic for all $i\in \{2,5\}$.
\item If $\sX=\sE_8$ then $\theta$ is $\{i\}$-domestic for all $i\in \{2,3,4,5\}$. 
\end{compactenum}
\end{lemma}

\begin{proof} Throughout the proof we repeatedly use Proposition~\ref{prop:proj} without reference.

(1) If there exists a type $2$ vertex $v$ mapped to an opposite, then $\theta_v$ is an automorphism of an $\sA_1\times\sA_2$ building acting as a duality on the $\sA_2$ component. Hence there is a type $\{2,3,4\}$ simplex mapped to an opposite by~$\theta$. If $\sigma$ is the type $\{3,4\}$ subsimplex of this simplex then $\theta_{\sigma}$ is a domestic duality of a large $\sA_2$ building, a contradiction. The argument for type $3$ vertices is dual.

(2) Suppose that there exists a type $3$ vertex $v$ mapped to an opposite vertex. Then $\theta_v$ acts as a duality on the $\sA_4$ component of $\Res(v)$, and hence by Theorem~\ref{thm:Alarge} there is a type $\{2,3,4,5,6\}$ simplex of $\Delta$ mapped to an opposite by $\theta$. If $\sigma$ is the type $\{2,4,5,6\}$ subsimplex of this simplex then $\theta_{\sigma}$ is a domestic duality of a large $\sA_2$ building, a contradiction. The argument for type $5$ vertices is dual.

It there is a type $2$ vertex mapped to an opposite, then considering the type $\sA_5$ residue we see that there is a type $3$ vertex mapped to an opposite (because the induced duality is either a symplectic polarity, or is not domestic), contradicting the previous paragraph. If there is a type $4$ vertex  $v$ mapped to an opposite, then $\theta_v$ acts as a duality on each of the $\sA_2$ components of $\Res(v)$, and hence there is a type $3$ vertex mapped to an opposite, a contradiction. 

(3) Suppose that there exists a type $4$ vertex $v$ mapped to an opposite by $\theta$. Then $\theta_v$ is an automorphism of an $\sA_2\times\sA_1\times\sA_2$ building interchanging the two $\sA_2$ components, and hence by Lemma~\ref{lem:AnAn} there is a type $\{1,3,4,5,6\}$ simplex of $\Delta$ mapped to an opposite by $\theta$. If $\sigma$ is the type $\{1,3,5,6\}$ subsimplex of this simplex then $\theta_{\sigma}$ is a domestic duality of a large $\sA_2$ building, a contradiction.

Suppose that there exists a type $\{3,5\}$ simplex $\sigma$ mapped to an opposite simplex. Then $\theta_{\sigma}$ is an automorphism of an $\sA_1\times \sA_1\times\sA_1$ building interchanging two of the components. Therefore there is a type $\{1,3,5,6\}$ simplex $\sigma'$ mapped to an opposite. Then $\theta_{\sigma'}$ is a domestic duality of a large $\sA_2$ building, a contradiction.

(4) See Example~\ref{ex:E7}.
%

(5) If there exists a type $2$ vertex~$v$ mapped to an opposite then $\theta_v$ is a symplectic polarity of an $\sA_7$ building, and hence there exists a type $\{2,3,5,7\}$ simplex mapped to an opposite by~$\theta$. If $v'$ is the type $3$ vertex of this simplex then $\theta_{v'}$ is an automorphism of an $\sA_1\times \sA_6$ building acting as a duality on the large $\sA_6$ component, and hence there is a type $\{2,3,4,5,6,7,8\}$ simplex mapped to an opposite by $\theta$. If $\sigma$ is the type $\{2,4,5,6,7,8\}$ subsimplex of this simplex, we see that $\theta_{\sigma}$ is a domestic duality of a large $\sA_2$ building, a contradiction. 

The remaining cases are similar: If there is a type $3$ or $5$ vertex mapped to an opposite then by considering the residue we see that there is a type $2$ vertex mapped to an opposite, contradicting the previous paragraph. If there is a type $4$ vertex mapped to an opposite then in the residue we see that there is a type $3$ vertex mapped to an opposite, contradicting the previous sentence. 
\end{proof}

To prove cappedness of automorphisms in large $\sF_4$ and $\sE_7$ buildings we require two additional, and  nontrivial, facts. Note that the second statement below applies also to small $\sE_7$ buildings.

\begin{prop}\label{prop:summary2}\leavevmode
\begin{compactenum}[$(1)$]
\item Let $\theta$ be a collineation of a large $\sF_4$ building. If $\theta$ is $\{1,4\}$-domestic then $\theta$ is either $\{1\}$-domestic or $\{4\}$-domestic. 
\item Let $\theta$ be a collineation of a thick $\sE_7$ building. If $\theta$ is $\{3,7\}$-domestic then $\theta$ is either $\{3\}$-domestic or $\{7\}$-domestic. 
\end{compactenum}
\end{prop}
We temporarily postpone the proof of Proposition~\ref{prop:summary2}, and first show how the main theorem readily follows. 

\begin{thm}
Every automorphism of a large building of type $\sE_n$ or $\sF_4$ is capped.  
\end{thm} 

\begin{proof}
Every non-domestic automorphism is automatically capped, and hence it suffices to consider only nontrivial domestic automorphisms~$\theta$. Let $J$ be the union of all $J'\subseteq S$ such that there is a type $J'$ simplex mapped to an opposite simplex by~$\theta$.

Let $\Delta$ be a large building of type $\sF_4$. By Lemma~\ref{lem:F41} we may assume that $\theta$ is a nontrivial domestic collineation, and thus by Lemma~\ref{lem:exceptional} $\theta$ is $\{i\}$-domestic for $i=2,3$. Thus $J=\{1\}$, $\{4\}$, or $\{1,4\}$. In the first two cases $\theta$ is trivially capped, and if $J=\{1,4\}$ then Proposition~\ref{prop:summary2}$(1)$ implies that there is a type $\{1,4\}$ simplex mapped to an opposite, and so $\theta$ is capped. 

Let $\theta$ be a domestic duality of a large $\sE_6$ building. By Lemma~\ref{lem:exceptional} $\theta$ is $\{i\}$-domestic for $i\in  \{2,3,4,5\}$. If $\theta$ maps a type $6$ vertex $v$ to an opposite then $\theta_v$ is a duality of a $\sD_5$ building, and hence maps a type $1$ vertex to an opposite by Proposition~\ref{prop:DDual}. It follows that $J=\{1,6\}$, and $\theta$ is capped. 

Let $\theta$ be a nontrivial domestic collineation of a large $\sE_6$ building. By Lemma~\ref{lem:exceptional} $\theta$ is $\{4\}$-domestic and $\{3,5\}$-domestic. If $J=\{2\}$ then $\theta$ is trivially capped. If $\{1,6\}\subseteq J$ then there is a type $\{1,6\}$ simplex $\sigma$ mapped to an opposite simplex. Then $\theta_{\sigma}$ is a duality of the type $\sD_4$ building $\Res(\sigma)$, and hence by Proposition~\ref{prop:DDual} $J=\{2,1,6\}$ and $\theta$ is capped. 

See Example~\ref{ex:E7} for the details for collineations of large $\sE_7$ buildings. Note that Proposition~\ref{prop:summary2}(2) is invoked at this stage.

Let $\theta$ be a nontrivial domestic collineation of a large $\sE_8$ building. By Lemma~\ref{lem:exceptional} we have $J\subseteq \{1,6,7,8\}$. Basic residue arguments show that if $1\in J$ then $\{1,8\}\subseteq J$, and if either $6\in J$ or $7\in J$ then $J=\{1,6,7,8\}$. Thus $J=\{8\}$, $\{1,8\}$, or $\{1,6,7,8\}$, and the residue arguments show that in each case $\theta$ is capped. 
\end{proof}

\subsection{Proof of Proposition~\ref{prop:summary2}}

It remains to prove Proposition~\ref{prop:summary2}. 

\subsubsection{Proof of Proposition~\ref{prop:summary2}(1)}

We now prove Proposition~\ref{prop:summary2}(1). In fact it is useful to prove the following slightly stronger version, which will also be useful for small buildings in later work (by Lemma~\ref{lem:exceptional} the $\{i\}$-domesticity assumption for $i=2,3$ is superfluous for large $\sF_4$ buildings). 

\begin{lemma}\label{lem:F42}
Let $\theta$ be a collineation of a thick $\sF_4$ building and suppose that $\theta$ is $\{i\}$-domestic for $i=2,3$. If $\theta$ is $\{1,4\}$-domestic then $\theta$ is either $\{1\}$-domestic or $\{4\}$-domestic. 
\end{lemma}

Recall that $\sF_4$ buildings are metasymplectic spaces, with vertices of types $1,2,3,4$ being the points, lines, planes, and symplecta of the space (see, for example, Chapter 18 of \cite{Shu:10}). Hence throughout this section we assume that $\theta$ is line-domestic and plane-domestic, and that $p$ is a point mapped onto an opposite point. We will show that no symplecton is mapped onto an opposite symplecton.

\begin{lemma}\label{lemma1F4}
The map $\theta_p$ is the identity. 
\end{lemma}

\begin{proof}
We assume that no lines or planes are mapped to an opposite, and the assumption of $\{1,4\}$-domesticity says that every symplecton containing $p$ is domestic. Thus $\theta_p$ has only domestic elements, and hence is the identity. 
\end{proof}

\begin{lemma}\label{lemma2F4}
Let $\pi$ be a plane containing $p$, and let $L$ be the projection of $p^\theta$ onto $\pi$. Then $L^\theta$ is the projection of $p$ onto $\pi^\theta$, every point $x$ of $L$ is mapped onto the unique point on $L^\theta$ that is collinear to $x$, and no point of $\pi\setminus L$ is domestic.
\end{lemma}

\begin{proof}
Set $p'=p^\theta$ and $\pi'=\pi^\theta$. Let $L'$ be the projection of $p$ onto $\pi'$. First we claim that  $L'=L^\theta$. Suppose not. Let $q=L\cap {L'}^{\theta^{-1}}$ and pick a point $r$ on the line $pq$ distinct from $p$ and $q$. Then $r':=r^\theta$ is not contained in $L'$ and so is opposite $p$, and hence also opposite $r$.  Consequently $\theta_r$ exists and is the identity. Notice that $q^\theta=L^\theta\cap L'$, and since $q^\theta$ belongs to the unique line of $\pi'$ not opposite $pq$, the point $q^\theta$ is collinear to $q$ and we have the path $p\perp q\perp q^\theta\perp p^\theta$. Now take any line $K$ in $\pi$ containing $r$ and not containing $p$. Put $u=K\cap L$. The projection of $K$ onto $p'$ is the line $p'u'$, where $u$ is the unique point on $L'$ collinear to $u$. Since $\theta_p$ and $\theta_r$ are the identiy, we have that $pu^\theta=p'u'$ and $(ru)^\theta=r'u'$ and so $u^\theta=(pu\cap ru)^\theta =(pu)^\theta\cap(ru)^\theta=p'u'\cap r'u'=u'$. Now it is also clear that $L'=L^\theta$ and the other statements also follow easily.
\end{proof}

\begin{lemma}\label{lemma3F4}
Let $\Sigma$ be a symplecton containing $p$, and let $p_\Sigma$ be the projection of $p^\theta$ onto $\Sigma$. Then $p^\theta_\Sigma=p_{\Sigma}$ and the set of points of $\Sigma$ mapped onto an opposite coincides with the set $\Sigma\setminus p_\Sigma^\perp$. Also, every point of $p_\Sigma^\perp\cap\Sigma$ is mapped onto a collinear point, except for $p_\Sigma$, which is fixed.
\end{lemma}

\begin{proof}
Noticing that, by the general theory (see \cite{Tit:74}), the projection of $p^\theta$ onto $p^\perp\cap\Sigma$ is equal to the projection of $p_\Sigma$ onto $p^\perp\cap\Sigma$, Lemma~\ref{lemma2F4} implies that all points of $(p^\perp\cap\Sigma)\setminus p^\perp_\Sigma$ are non-domestic. Now let $w$ be any point of $\Sigma$ not collinear with $p_\Sigma$. 
 If $w\perp p$, then by Lemma~\ref{lemma2F4}, $w$ is non-domestic. In the other case, let $M$ be any line in $\Sigma$ through $w$. It is easy to see that there exists some line $M'$ through $p$ in $\Sigma$ which is, with self-explaining terminology, $\Sigma$-opposite $M$. Then some point $z$ of $M'\setminus p_\Sigma^\perp$ is collinear to some point $z'$ of $M\setminus p_\Sigma^\perp$ (since there are at least three points per line). Lemma~\ref{lemma2F4} again yields that $z$ is non-domestic, similarly, letting $z$ play the role of $p$,  $z'$ is non-domestic, and similarly again, $w$ is non-domestic. 

Now let $q$ be the preimage of $p_\Sigma$; so $q^\theta=p_\Sigma$, and notice that $q\in\Sigma$ since $p_\Sigma\in\Sigma^\theta$. Let $x\in q^\perp\cap\Sigma$ be arbitrary. We claim that $x$ is domestic. If $x=q$, this is obvious, so assume $x\neq q$. Then $x^\theta\perp q^\theta=p_\Sigma\in\Sigma$. By a general property of metasymplectic spaces (see \cite{Shu:10}), $x^\theta$ is collinear to all points of a line $L\subseteq\Sigma$, and so $x^\theta\perp y\perp x$, for some point $y\in L$. Hence the claim. So $q^\perp\cap\Sigma\subseteq p_\Sigma^\perp\cap\Sigma$, yielding $q=p_\Sigma$ as an obvious general property of polar spaces. The rest of the statement is now obvious, in view of Lemma~\ref{lemma2F4}.   
\end{proof}

For the final lemma, we note that opposite symplecta do not contain respective points which are collinear.

\begin{lemma}
No symplecton is mapped onto an opposite.
\end{lemma}

\begin{proof}
By the $\{1,4\}$-domesticity assumption, no non-domestic symplecton is incident with a non-domestic point. Let $\Sigma$ be any non-domestic symplecton. Then there are two possibilities (see Chapter 18 of \cite{Shu:10}).
\begin{enumerate}[$(1)$]
\item \emph{$p$ is collinear to all points of a line $L\subseteq\Sigma$:} Let $\pi$ be the plane through $p$ and $L$. Then $L$ and the projection of $p^\theta$ onto $\pi$ have at least one point $x$ in common, which, by Lemma~\ref{lemma2F4}, is mapped onto a collinear point $x^\theta$. Since $x^\theta\in\Sigma^\theta$, the latter is not opposite $\Sigma\ni x$.  
\item \emph{$p$ is not collinear to any point of $\Sigma$:} Let $y$ be the projection of $p$ onto~$\Sigma$. Let $u$ be the projection of $p^\theta$ onto the unique symplecton $\Omega$ containing $p$ and $y$. If $y\notin u^\perp$, then $y$ is non-domestic, a contradiction as also $\Sigma$ is non-domestic and as $\theta$ is $\{1,4\}$-domestic. If $y\in u^\perp$, then by Lemma~\ref{lemma3F4}, $y$ is either fixed (in which case $y\in\Sigma^\theta$, contradicting the fact that $\Sigma$ is non-domestic) or $y\perp y^\theta$, again implying that $\Sigma$ is domestic.
\end{enumerate}
Hence in all cases we reached a contradiction, implying that no non-domestic symplecton exists. Hence $\theta$ is either $1$-domestic or $4$-domestic.
\end{proof}

\subsubsection{Proof of Proposition~\ref{prop:summary2}(2)}

Here we need to argue in the strong parapolar space $\mathfrak{P}$ of type $\mathsf{E_{7,7}}$ associated to a building $\Delta$ of type $\mathsf{E_7}$, where vertices of type $7$ of $\Delta$ correspond to points of $\mathfrak{P}$, and where vertices of type $3$ of $\Delta$ correspond to $5$-dimensional maximal singular subspaces of $\mathfrak{P}$ (see, for example, \cite{Shu:10}). 

Before embarking on the proof of the lemma, we need to recall some basics about $\mathfrak{P}$, mainly concerning the possible mutual positions of objects like points, symplecta and maximal singular $5$-spaces. 

First recall that $\mathfrak{P}$ is a strong parapolar space of diameter $3$, which implies that two distinct points $p,q$ are either collinear (and the unique line passing through them is denoted $pq$), or opposite (this is distance 3), or not collinear and contained in a unique symplecton which we denote by $\Sigma(p,q)$. In the latter case we say that \emph{$p$ is symplectic to $q$}. 

Secondly recall that the maximal singular subspaces come in two flavours: there are $5$-dimensional maximal singular subspaces, which we shall call $5$-spaces, and $6$-dimensional ones, which we shall call $6$-spaces. The $5$-dimensional subspaces of the $6$-spaces shall be referred to as $5'$-spaces. Every other singular space shall be referred to as an $i$-space, where $i$ is its dimension. 

\begin{fact}\label{fact1E7}
Let $p$ be a point and let $\Sigma$ be a symplecton. Then exactly one of the following possibilities occurs.
\begin{compactenum}[$(i)$]
\item $p\in\Sigma$ (we say $p$ and $\Sigma$ are incident);
 \item the set of points of $\Sigma$ collinear with $p$ is a $5'$-space $U$, and $U$ and $p$ are contained in a unique $6$-space. All points of $\Sigma\setminus U$ are symplectic to $p$. We say that $p$ is close to $\Sigma$, or that $p$ and $\Sigma$ are close.
 \item There is a unique point $x$ of $\Sigma$ collinear with $p$; every point of $\Sigma$ collinear with $x$ is symplectic to $p$, every other point of $\Sigma$ is opposite $p$. We say that $p$ is far from $\Sigma$, or that $p$ and $\Sigma$ are far. 
\end{compactenum} 
\end{fact}
 
 \begin{fact}\label{fact2E7}
Let $\Sigma$ and $\Sigma'$ be two distinct symplecta. Then exactly one of the following holds.
\begin{compactenum}[$(i)$]
\item $\Sigma$ and $\Sigma'$ intersect in a $5$-space. We say that $\Sigma$ and $\Sigma'$ are adjacent.
\item $\Sigma$ and $\Sigma'$ intersect in a line. We say that $(\Sigma,\Sigma')$ is a polar pair.
\item $\Sigma\cap\Sigma'=\emptyset$ and there exists a unique symplecton adjacent to both $\Sigma$ and $\Sigma'$. We say that $(\Sigma,\Sigma')$ is a special pair. 
\item Every point of $\Sigma$ is collinear to a unique point of $\Sigma'$. Then we say that $\Sigma$ is opposite $\Sigma'$. 
\end{compactenum} 
 \end{fact}
 
 \begin{fact}\label{fact3E7}
Let $p$ be a point and let $U$ be a $5$-space. Then exactly one of the following occurs.
\begin{compactenum}[$(i)$]
\item $p\in U$. 
\item $p\notin U$ and $p$ is collinear to all points of a (unique) $4$-space in $U$.
\item $p$ is collinear to all points of a unique plane in $U$. 
\item $p$ is collinear with a unique point of $U$.
\item $p$ is symplectic to every point of a unique $4$-space of $U$ and opposite the other points of $U$. 
\end{compactenum} 
\end{fact}
 
From now on we assume that the point $p$ of $\mathfrak{P}$ is non-domestic, and put $p'=p^\theta$. We have to show that every $5$-space is domestic. Note that two $5$-spaces are opposite if and only if every point of one $5$-space is symplectic to all points of a unique $4$-space of the other $5$-space and opposite the other points of the other $5$-space. In particular, if two $5$-spaces contain two respective collinear points, then they are not opposite.

We begin with an easy lemma.

\begin{lemma}\label{lemma1E7}
The mapping $\theta_p$ is a symplectic polarity in the residue of $p$. 
\end{lemma}
 
 \begin{proof}
The $\{3,7\}$-domesticity immediately implies that no $4$-space of the residue of $p$ is mapped under $\theta_p$ to an opposite line of the residue. Thus $\theta_p$ is a duality of the $\sE_6$ residue, domestic on type $3$ vertices, and it follows from \cite{HVM:12} that $\theta_p$ is a symplectic polarity. (In fact, as mentioned above, the main results of \cite{HVM:12} only apply to large $\sE_6$ buildings since the existence of the exceptional domestic duality of the Fano plane was overlooked in \cite{HVM:12}. In the case of the small building $\sE_6(2)$ it is easy to see that every domestic duality is either a symplectic polarity or an exceptional domestic duality, see \cite{PVM:17b} for details, and since the latter is not $\{3\}$-domestic the result claimed above holds in this building too.) 
 \end{proof}
 
The previous lemma implies that there are two types of symplecta through $p$: (1) a symplecton $\Sigma$ of the first kind satisfies $\Sigma^{\theta_p}\subseteq\Sigma$, (2) a symplecton $\Sigma$ of the second kind satisfies $\Sigma^{\theta_p}\cap\Sigma=\{p\}$. This also defines two types of lines through $p$: (1) lines of the first kind are lines $L$ for with $L\subseteq L^{\theta_p}$ and (2) lines $L$ of the second kind satisfy $L\cap L^{\theta_p}=\{p\}$. 

The well know general properties of symplectic polarities in geometries of type $\mathsf{E}_{6,1}$ imply the following facts (see \cite{HVM:12}).

\begin{fact}\label{fact4E7} As above, let $p$ be a non-domestic point of $\mathfrak{P}$. 
\begin{compactenum}[$(i)$]
\item Let $L$ be a line through $p$ such that $L\subseteq \Sigma:=L^{\theta_p}$. Then all lines of $\Sigma$ through $p$ contained in a common plane with $L$ are of the first kind, while all other lines through $p$ in $\Sigma$ are of the second kind.\item If a plane through $p$ contains at least two lines through $p$ of the first kind, then all lines of that plane through $p$ are of the first kind. Dually, if two distinct symplecta sharing a $5$-space through $p$ are of the first kind, then all symplecta containing that $5$-space are of the first kind. \item Every plane through $p$ contains at least one line through $p$ of the first kind. Dually, every $5$-space through $p$ is contained in at least one symplecton of the first kind.  
\end{compactenum}
\end{fact}

\begin{lemma}\label{lemma2E7}
Let $L\ni p$ be any line of the first kind. Let $q$ be the unique point of $L$ not opposite~$p'=p^{\theta}$. Let $q'$ be the projection of $p$ onto $L^\theta$. Then $q\perp q'=q^\theta$. In particular, all points of $L\setminus\{q\}$ are non-domestic.
\end{lemma}
 
\begin{proof}
Since $L$ is of the first kind, the projection $\Sigma$ of $L^\theta$ onto $p$ contains the line $L$ and intersects $L^\theta$ in the point $q'$. We see that inside $\Sigma(p,q')$, there is a unique point $q$ of $L$ collinear with $q'$, and it is the projection of $p'$ onto $L$, since it is symplectic to $p'$. 

Consider any plane $\pi$ in $\Sigma(p,q')$ containing $L$. Let $K$ be a line of $\pi$ containing $p$ and distinct from $L$. Since $\theta_p$ is a polarity (and hence has order 2), the image $K^{\theta_p}$ contains the image $(\Sigma(p,q')^{\theta_p})^{\theta_p}=L$. As before, we obtain points $r$ on $K$ and $r'$ on $K^\theta$ with $r$ and $r'$ collinear. The symp $\Sigma(p,r')$ contains $\pi$. Obviously $L^\theta$ and $K^\theta$ are contained in a plane, namely $\pi^\theta$. It is the projection of $\Sigma(p,q')\cap\Sigma(p,r')$ (which is a $5$-space) onto $p^\theta$. If the line $(qr)^\theta$ coincides with $q'r'$, then the lemma follows. Suppose now $(qr)^\theta\neq q'r'$. Then $(qr)^\theta$ intersects $q'r'$ in some point, and without loss of generality we may assume $r'=  (qr)^\theta\cap q'r'$. Let $u$ be any point in $K\setminus\{p,r\}$. Then $u$ and $u^\theta$ are opposite and applying the foregoing to $u$ instead of $p$ yields $(uq)^\theta=u^\theta q'$. Hence $q^\theta=(pq\cap uq)^\theta=p^\theta q'\cap u^\theta q'=q'$, implying $(qr)^\theta=q'r'$ after all.  
\end{proof}

\begin{lemma}\label{lemma3E7}
Let $\Sigma$ and $\Sigma'$ be two symplecta intersecting in a line $L$. Let $u\in\Sigma$ and $u'\in\Sigma'$ be two points. Then $u$ and $u'$ are opposite if and only if no point on $L$ is collinear with both. Also, if $u$ and $u'$ are collinear with the same unique point of $L$, then they are symplectic.
\end{lemma} 
 
\begin{proof}
Clearly if some point on $L$ is collinear with both $u$ and $u'$, then $u$ and $u'$ cannot be opposite as they have distance at most~$2$. Now suppose $u$ and $u'$ are not opposite and at the same time not collinear a common point on $L$. Then they have a unique projection $v$ and $v'$, respectively, on $L$, with $v\neq v'$. If $u\perp u'$, then $\Sigma=\Sigma(u,v')=\Sigma(u',v)$ would contain $u'$, a contradiction. Suppose $u$ and $u'$ are symplectic. Let $z$ be a point collinear to both. Then $z\notin\Sigma\cup\Sigma'$. If $u$ were the only point of $\Sigma$ collinear with $z$, then by Fact~\ref{fact1E7}, $v'$ is opposite $z$, clearly absurd as $z\perp u'\perp v'$. Hence $z$ is close to $\Sigma$ and collinear to all points of some $5'$-space $U$ of $\Sigma$.  As $v'$ is not collinear with $u$, the set of points of $U$ collinear with $v'$ is a $4$-space $W$.  Then $W$ and $u'$ belong to $\Sigma(v',z)$. Hence some $w\in W$ is collinear with both $v$ and $u'$ and thus belongs to $\Sigma(u',v)=\Sigma'$, a contradiction. 

Now suppose that $u$ and $u'$ are collinear with the same point $r$ on $L$ and not with any other point of $L$. Suppose $u\perp u'$. Then $u'$ is collinear to all points of a $5'$-space of $\Sigma$, and at least one such point $t\notin L$ is collinear to all points of $L$. If $r'$ is such point distinct from $r$, then $t\in\Sigma(r',u')=\Sigma'$, a contradiction.
\end{proof}

\begin{cor}\label{cor1E7}
Let $K$ be a line of $\mathfrak{P}$ all of whose points are symplectic to $p$. Then there exists a unique $5$-space $U$ such that the set of symplecta containing $U$ coincides with the set of symplecta containing $p$ and a point of $K$.
\end{cor}

\begin{proof}
Take two distinct points $x,y$ on $K$ and assume $\Sigma(p,x)\cap\Sigma(p,y)=L$ is a line. Since neither of $x,y$ is collinear to $p$, the points $x$ and $y$ are both collinear to a unique respective point of $L$. Lemma~\ref{lemma3E7} implies that $x$ and $y$ are either opposite or symplectic, a contradiction. Hence Fact~\ref{fact2E7} yields a unique $5$-space $U=\Sigma(p,x)\cap\Sigma(p,y)$. 

Consider an arbitrary point $z$ on $K$. We claim that $U\subseteq\Sigma(p,z)$. Indeed, $x$ is collinear to all points of the $5'$-space of $\Sigma(p,y)$ generated by $x^\perp\cap U$ and $y$. Hence $x^\perp\cap U=y^\perp\cap U=:W$, and $W$ and $K$ are contained in a unique $6$-space. This implies that $z$ is collinear to all points of $W$. This also implies $W\subseteq\Sigma(p,z)$, and the claim follows.

Now, conversely, suppose some symplecton $\Sigma$ contains $U$ en assume that $\Sigma$ is disjoint from $K$. The subspace $W$ and the line $K$ generate a $6$-space $A$; the point $x$ is collinear to all points of a $5'$-space $B$ and $A\cap B$ contains $x$ and $W$, hence is a $5'$-space. But then $\Sigma(a,b)$, with $a\in A\setminus B$ and $b\in B\setminus A$, contains $A$ and $B$, which is totally absurd. 
\end{proof}

\begin{lemma}\label{lemma4E7}
Let $L$ be a line of the first kind through $p$. Let $K$ be the unique line concurrent with both $L$ and $L^\theta$. then $K^\theta=K$. 
\end{lemma}

\begin{proof}
Let $q$ be the projection of $p^\theta$ onto $L$ and $q'$ the projection of $p$ onto $L^\theta$. Then we already know by Lemma~\ref{lemma2E7} that $q^\theta=q'$. So, if $K\neq K^\theta$, then $K^\theta$ is a line in $\Sigma'$ (the projection of $L$ onto $p^\theta$) intersecting $\Sigma$ (the projection of $L^\theta$ onto $p$) in the point $q'$ (note that $\Sigma\cap\Sigma'=K$). Let $u$ be any point of $K\setminus\{q,q'\}$ and let $y$ be any point in $\Sigma$ collinear to $u$ but not to $q$. Then $y^\theta$ is collinear to $u^\theta$ on $K^\theta$, but not to $q'$. Hence there is a unique point on $(uy)^\theta$ collinear to $q$, and we may without loss of generality assume it is $y^\theta$. Letting $y$ play the role of $p$ in the foregoing, we conclude that $u^\theta = q$. So $K$ is fixed after all.  
\end{proof}

We can already rule out some possibilities for a non-domestic $5$-space.

\begin{lemma}\label{lemma5E7}
If $U$ is a non-domestic $5$-space of $\theta$, then $p$ is collinear with at most one point of $U$.
\end{lemma}

\begin{proof}
By the assumption of $\{3,7\}$-domesticity, $p\notin U$. Suppose $p$ is collinear with all points of a line $M$ of $U$. By Fact~\ref{fact4E7}, the plane $\pi$ generated by $p$ and $M$ contains a line $L$ through $p$ of the first kind. Then $L\cap U$ either is a non-domestic point, a contradiction with the $\{3,7\}$-domesticity, or gets mapped under $\theta$ to a collinear point, and then $U^\theta$ cannot be opposite $U$. 
\end{proof}

Now we need two lemmas for polar spaces. We prove them in the most general case, even though we only need them for polar spaces of ranks 4 and 6. 

\begin{lemma}\label{lemma6E7}
Suppose $\Omega$ is a polar space, $L$ a line of $\Omega$, $x$ a point of $\Omega$ not on $L$, and $y$ a point of $\Omega$ on $L$. Then there exists a point collinear with or equal to $x$, collinear with $y$, but not collinear with all points of $L$.   
\end{lemma}

\begin{proof}
If $x\perp y$, then considering the residue in $y$, we find a line through $y$ all of whose points are collinear with $x$ but not with any point of $L$ distinct from $y$. Any point on such line does the job. If $x$ is not collinear with $y$, then we consider the projection of $x$ onto any line through $y$ not contained in a common plane with $L$.
\end{proof}

\begin{lemma}\label{lemma8E7}
Suppose $\Omega$ is a polar space, $L$ a line of $\Omega$, $x$ a point of $\Omega$ collinear with exactly one point   $y$ of $L$, and $K$ is a line of $\Omega$ not coplanar with $x$, not intersecting $L$ and such that the unique line $M$ containing $x$ and intersecting $K$ in a point $z$ is not coplanar with $y$.  Then there exists a point collinear with all points of $K$, collinear with $y$, and not coplanar with $L$. 
\end{lemma}

\begin{proof}
The assumptions imply that $M$ and $L$ are $\Omega$-opposite lines. Let $z'$ be the projection of $z$ onto $L$. Also, $xy$ is $\Omega$-opposite $K$. Let $y'$ be the projection of $y$ onto $K$. If $z'$ is not collinear to $y'$, then $y'$ satisfies our conditions. So we may assume that $y'$ is coplanar with $L$ (and this in particular implies that we may assume that the rank of $\Omega$ is at least 3). Note that the condition that $K$ and $L$ are disjoint implies that $y'\neq z'$. 

Now there is a  plane $\pi$ through $y$ $\Omega$-opposite the plane $\<y,y',z\>$ in the residue of $y$. That yields a line $L'$ in $\alpha$ coplanar with $y'$ but not with $z'$. No point of $L'\setminus\{y\}$ is coplanar with $L$ and since $y$ and $z$ are not collinear, there is a unique point $u\in L'\setminus\{y\}$ collinear with $z$, and hence coplanar with $K$. The point $u$ satisfies our requirements.    
\end{proof}

Now we can finish the proof of the $\{3,7\}$-lemma for buildings of type $\mathsf{E_7}$.

\begin{lemma}\label{lemma9E7}
If $U$ is a non-domestic $5$-space of $\theta$, then $p$ is not collinear with any point of $U$.
\end{lemma}

\begin{proof}
Suppose $p$ is collinear to some point $z$ of $U$ and set $M:=pz$. Then by Lemma~\ref{lemma5E7}, $z$ is unique with this property. Hence all other points of $U$ are symplectic to $p$. Considering a line in $U$ not through $z$, Corollary~\ref{cor1E7} combined with the third assertion of Fact~\ref{fact4E7} (dual form) implies that there exists a symplecton $\Sigma$ of the first kind through $p$ and some line $K$ of $U$ containing $z$.  Let $R$ be the projection of $\Sigma^\theta$ onto $p$ (then $R$ is a line of the first kind), set $y=\Sigma^\theta\cap R$, and let $L$ be the line containing $y$ and $y^\theta$. Then Lemma~\ref{lemma4E7} asserts that $L^\theta= L$. Hence, if $K$ would intersect $L$, then some point of $U$ is mapped onto a collinear one, contradicting non-domesticity of $U$. So we may assume that $K$ and $L$ do not intersect.

If $M$ and $y$ were coplanar, then the first assertion of Fact~\ref{fact4E7} would imply that $M$ is of the first kind. Hence either $z^\theta\perp z$, contradicting $U$ being non-domestic, or $z$ is non-domestic itself, contradicting the $\{3,7\}$-domesticity assumption. So we may assume that $M$ and $y$ are not coplanar. 

Hence all assumptions of Lemma~\ref{lemma8E7} are satisfied (where the current point $p$ plays the role of the point $x$ in Lemma~\ref{lemma8E7}) and we conclude that there is a point $u$ of $\Sigma$ not coplanar with $L$ but collinear with $y$, and coplanar with $K$. Lemma~\ref{lemma3E7} implies that $u\in\Sigma$ and $u^\theta\in\Sigma^\theta$ are opposite, hence $u$ is non-domestic. But now $u$ is collinear to all points of $K\subseteq U$, contradicting Lemma~\ref{lemma5E7}, where we interchange the roles of $u$ and $p$. 
\end{proof}

\begin{lemma}
Every $5$-space is $\theta$-domestic. That is, $\theta$ is $\{3\}$-domestic.
\end{lemma}

\begin{proof}
Suppose for a contradiction that some $5$-space $U$ is non-domestic. By Lemmas~\ref{lemma5E7} and~\ref{lemma9E7}, an in view of Fact~\ref{fact3E7}, there is a unique $4$-space $W\subseteq U$ all of whose points are symplectic to $p$, and such that all points of $U\setminus W$ are opposite $p$. Considering an arbitrary line in $W$, Corollary~\ref{cor1E7} combined with the third assertion of Fact~\ref{fact4E7} (dual form) implies that there exists a symplecton $\Sigma$ of the first kind containing $p$ and some point  $x$ of $W$. Let $L=\Sigma\cap\Sigma^\theta$ and let $y$ be the projection of $\Sigma^\theta$ onto its projection onto $p$, just like in the proof of Lemma~\ref{lemma9E7}. Note that, if $x\in L$, then $U$ is domestic (since $x$ is either fixed by $\theta$, of mapped onto a collinear point), a contradiction. Hence we may assume $x\notin L$. Then we can apply Lemma~\ref{lemma6E7} and find a point $u$ in $\Sigma$ collinear with $y$, not coplanar with $L$ and collinear with $x$. The point $u$ is non-domestic by Lemma~\ref{lemma3E7} and the fact that $y\neq y^\theta\in L$. But $u$ is collinear to the point $x$ of $U$, contradicting Lemma~\ref{lemma9E7} by interchanging the roles of $p$ and $u$.    
\end{proof}

This concludes the proof of the fact that $\{3,7\}$-domesticity in any thick building of type $\mathsf{E_7}$ implies either $3$-domesticity, or $7$-domesticity. Hence the proof of Theorem~\ref{thm:main} is complete.

\bibliographystyle{plain}

\end{document}